\documentclass[11pt]{amsart}

\usepackage{epigamath}

\usepackage[english]{babel}

\numberwithin{equation}{section}

\usepackage{enumitem}

\usepackage{microtype}

\usepackage{amsmath, amssymb, extarrows, mathtools, xparse, bm, mathdots}

\usepackage{amsthm, thmtools}

\usepackage{tikz-cd}

\usepackage{array}
\usepackage{makecell}
\setcellgapes{5pt}
\makegapedcells

\usepackage[capitalise, noabbrev]{cleveref}

\usepackage{csquotes}

\newtheorem{theorem}{Theorem}[section]
\newtheorem{corollary}[theorem]{Corollary}
\newtheorem{lemma}[theorem]{Lemma}
\newtheorem{proposition}[theorem]{Proposition}

\theoremstyle{definition}
\newtheorem{definition}[theorem]{Definition}

\theoremstyle{remark}
\newtheorem{remark}[theorem]{Remark}
\newtheorem*{question}{Question}

\usepackage{commands}

\newcommand{\lra}{\longrightarrow}

\EpigaVolumeYear{9}{2025} \EpigaArticleNr{5} \ReceivedOn{November 2, 2022}
\InFinalFormOn{July 24, 2024}
\AcceptedOn{August 10, 2024}

\title{The EKOR stratification on the Siegel modular variety with parahoric level structure}
\titlemark{EKOR stratification on the Siegel modular variety}
\author{Manuel Hoff}
\address{Fakult\"at f\"ur Mathematik, Universit\"at Bielefeld, 33501 Bielefeld, Germany}
\email{manuel.hoff@uni-bielefeld.de}

\authormark{M.~Hoff}

\AbstractInEnglish{We study the arithmetic geometry of the reduction modulo $p$ of the Siegel modular variety with parahoric level structure.
  We realize the EKOR stratification on this variety as the fibers of a smooth morphism into an algebraic stack parametrizing homogeneously polarized chains of certain truncated displays.}



\acknowledgement{This work was funded by the Deutsche Forschungsgemeinschaft (DFG, German Research Foundation), Graduiertenkolleg 2553 and Project-ID 491392403 - TRR 358.}

\begin{document}

%%%%%%%%%%%%%%%%%%%%%%%%%%%%%%%
% Title page
%%%%%%%%%%%%%%%%%%%%%%%%%%%%%%%

%\removeabove{}
%\removebetween{}
%\removebelow{}

\maketitle

\begin{prelims}

\DisplayAbstractInEnglish

\bigskip

\DisplayKeyWords

\medskip

\DisplayMSCclass

%\bigskip

%\languagesection{Fran\c{c}ais}

%\bigskip

%\DisplayTitleInFrench

%\medskip

%\DisplayAbstractInFrench

\end{prelims}

%%%%%%%%%%%%%%%%%%%%%
% Table of Contents
%%%%%%%%%%%%%%%%%%%%%

\newpage

\setcounter{tocdepth}{1}

\tableofcontents

%%%%%%%%%%%%%%%%%%%%%
% Content begins here
%%%%%%%%%%%%%%%%%%%%%

    \section{Introduction}

Fix a rational prime $p$, a positive integer $g$, an auxiliary integer $N \geq 3$ that is not divisible by $p$ and a non-empty subset $J \subseteq \ZZ$ with $J + 2g\ZZ = J$ and $-J = J$.
Then we are interested in studying the \emph{Siegel modular variety with parahoric level structure}
\[
    \calA_{g, J, N}
\]
over $\Zp$.
This is a quasi-projective scheme that parametrizes certain polarized chains of type $J$ of $g$\nobreakdash-dimensional Abelian varieties with full level $N$ structure; its moduli description was first given by De Jong \cite{de-jong}
and then in full generality by Rapoport and Zink \cite{rapoport-zink-96}.
The scheme $\calA_{g, J, N}$ is the prime example of an integral model of a \emph{Shimura variety} at a place of parahoric bad reduction.
While the generic fiber of $\calA_{g, J, N}$ is smooth, its special fiber typically is singular.
This is related to the fact that the $p$-torsion of an Abelian variety in characteristic $p$ is not \'etale.

Now fix an algebraic closure $\Fpbar$ of $\Fp$, let $\Zpbrev \coloneqq W(\Fpbar)$ be the ring of $p$-typical Witt vectors of $\Fpbar$, and write $\Qpbrev \coloneqq \Zpbrev[1/p]$ for its fraction field.
Denote by $\sigma$ the Frobenius on $\Qpbrev$.
Also fix a symplectic $\Qp$-vector space $V$ of dimension $2g$ and a self-dual lattice chain $(\Lambda_i)_{i \in J}$ in $V$.
Then there exists a natural map, called the \emph{central leaves map},
\[
    \Upsilon \colon \calA_{g, J, N} \roundbr*{\Fpbar} \lra \breve{K}_{\sigma} \backslash \GSp(V)(\Qpbrev);
\]
here $\breve{K} \subseteq \GSp(V)(\Qpbrev)$ is the stabilizer of the lattice chain $(\Lambda_i)_i$, and modding out by $\breve{K}_{\sigma}$ is notation for taking the quotient by the $\sigma$-conjugation action $g \colon x \mapsto g x \sigma(g)^{-1}$.
It is roughly given by sending a polarized chain of Abelian varieties to the twisted conjugacy class corresponding to the Frobenius $\Phi$ of the associated rational Dieudonn\'e module; see the work of Oort \cite{oort-foliations} and also He and Rapoport \cite{he-rapoport}.
The image of this map is given by $\breve{K}_{\sigma} \backslash X$, where
\[
    X = \bigcup_{w \in \Adm_{g, J}} \breve{K} w \breve{K} \subseteq \GSp(V)\left(\Qpbrev\right)
\]
is a finite union of double cosets that is indexed over the so-called \emph{admissible set} $\Adm_{g, J}$.

The fibers of the composition
\[
\lambda \colon \calA_{g, J, N} \roundbr*{\Fpbar} \stackrel{\Upsilon}\lra \breve{K}_{\sigma} \backslash X \lra \breve{K} \backslash X / \breve{K} \cong \Adm_{g, J}
\]
define a decomposition of $(\calA_{g, J, N})_{\Fpbar}$ into smooth locally closed subschemes, the so-called \emph{Kottwitz--Rapoport $($KR\,$)$ stratification}.
In fact Rapoport and Zink construct the following data: 

\begin{itemize}
    \item
    a flat projective scheme $\bbM^{\loc}$ over $\Zp$ with an action of the parahoric $\Zp$-group scheme $\calG$ for $\GSp(V)$ given by the lattice chain $(\Lambda_i)_i$, 

    \item
    a smooth morphism
    \[
        \calA_{g, J, N} \lra \squarebr*{\calG \backslash \bbM^{\loc}}
    \]
    that parametrizes the Hodge filtration in the de Rham cohomology of a polarized chain of Abelian varieties.
\end{itemize}
The scheme $\bbM^{\loc}$ is called the \emph{local model}; it parametrizes isotropic chains of $g$-dimensional subspaces $C_i \subseteq \Lambda_i$ and satisfies $\bbM^{\loc}(\Fpbar) = \breve{K} \backslash X$. One recovers the map $\lambda$  by evaluating the smooth morphism above on $\Fpbar$-points; see \cite[Section 7(ii)]{he-rapoport}.

He and Rapoport also consider the composition
\[
    \upsilon \colon \calA_{g, J, N} \roundbr*{\Fpbar} \lra \breve{K}_{\sigma} \backslash X \lra \breve{K}_{\sigma} \backslash (\breve{K} \backslash X), 
\]
where $\breve{K}_1 \subseteq \breve{K}$ is the pro-unipotent radical.
They observe that the codomain of $\upsilon$ is a finite set and call its fibers \emph{Ekedahl--Kottwitz--Oort--Rapoport $($EKOR\,$)$ strata}.

In the hyperspecial case $J = 2g \ZZ$,  the EKOR stratification is also just called the \emph{Ekedahl--Oort $($EO\,$)$ stratification} and was first considered by Oort \cite{oort-eo}.
Two points in $\calA_{g, 2g\ZZ, N}(\Fpbar)$ given by two (principally polarized) Abelian varieties lie in the same EO stratum if and only if their $p$-torsion group schemes are isomorphic.
Viehmann and Wedhorn \cite{viehmann-wedhorn} realize the EO stratification as the fibers of a smooth morphism from $(\calA_{g, 2g\ZZ, N})_{\Fp}$ into an algebraic stack that parametrizes $F$-zips in the sense of Moonen and Wedhorn \cite{moonen-wedhorn} with symplectic structure.

\begin{question}
    Is it possible for general $J$ to realize the map $\upsilon$, or maybe even $\Upsilon$, as a smooth morphism from $\calA_{g, J, N}$ into some naturally defined algebraic stack?
\end{question}

The existence of such a smooth morphism would in particular give a new proof of the smoothness of the EKOR strata and the closure relations between them.
More importantly, it also provides a tool that may be applied to further study the geometry of $\calA_{g, J, N}$.

\vspace*{1em}

Let us give an overview of results that have been obtained so far (and that we are aware of), also considering more general Shimura varieties than the Siegel modular variety:
\begin{itemize}
    \item
    Moonen and Wedhorn \cite{moonen-wedhorn} introduce the notion of an $F$-zip that is a characteristic $p$ analogue of the notion of a Hodge structure.
    Given an Abelian variety $A$ over some $\Fp$-algebra $R$, the de Rham cohomology $H^1_{\dR}(A/R)$ naturally carries the structure of an $F$-zip.
    If $R$ is perfect, then the datum of this $F$-zip is equivalent to the datum of the Dieudonn\'e module of the $p$-torsion $A[p]$.
    Pink, Wedhorn and Ziegler \cite{pink-wedhorn-ziegler} define a group-theoretic version of the notion of an $F$-zip.

    Viehmann and Wedhorn \cite{viehmann-wedhorn} define a moduli stack of $F$-zips with polarization and endomorphism structure in a PEL-type situation with hyperspecial level structure.
    They construct a morphism from the special fiber of the associated Shimura variety to this stack, parametrizing the EO stratification.
    They then show that this morphism is flat and use this to show that the EO strata are non-empty and quasi-affine and to compute their dimension.
    The smoothness of the EO strata was already shown by Vasiu \cite{vasiu-2}.

    Zhang \cite{zhang-eo} constructs a morphism from the special fiber of the Kisin integral model, see \cite{kisin-integral-models}, of a Hodge-type Shimura variety to a stack of group-theoretic zips.
    They then show that this morphism is smooth and thus gives an EO stratification with the desired properties.
    Shen and Zhang \cite{shen-zhang-eo} later generalize this to Shimura varieties of Abelian type.

    Shen, Yu and Zhang \cite{shen-yu-zhang} further generalize this to Shimura varieties of Abelian type at parahoric level, where the construction of integral models is due to Kisin and Pappas; see \cite{kisin-pappas-integral-models}.
    However, they only construct a morphism from each KR stratum of the special fiber of the Shimura variety into a certain stack of group-theoretic zips, parametrizing the EKOR strata contained in this KR stratum.
    Still, they show that this morphism is smooth, thus establishing the smoothness of the EKOR strata.

    Hesse \cite{hesse} considers an explicit moduli stack of polarized chains of $F$-zips and constructs a morphism from the Siegel modular variety into this stack.
    However, it appears that such a morphism is not well behaved
    in the general parahoric situation; see \Cref{rmk:chains-lau-zink}.

    \item
    Xiao and Zhu \cite{xiao-zhu} consider a perfect moduli stack of mixed characteristic local shtukas as well as restricted versions in a hyperspecial situation.
    For a Shimura variety of Hodge type, they construct a morphism from the perfection of its special fiber into the stack of local shtukas and show that it realizes the central leaves map.
    They state that the induced morphism to the moduli stack of restricted local shtukas is perfectly smooth, see \cite[Proposition 7.2.4]{xiao-zhu}, but the proof appears to be incomplete because the diagram used there is not commutative as required.

    Shen, Yu and Zhang \cite{shen-yu-zhang} again generalize this to the parahoric situation.
    Similarly to Xiao and Zhu, they construct a morphism from the perfected special fiber of the Shimura variety to their stack of local shtukas and state that the induced morphism to the moduli stack of restricted local shtukas is perfectly smooth, but their proof, see \cite[Theorem 4.4.3]{shen-yu-zhang}, employs a similar non-commutative diagram.
    This issue is addressed in the errata \cite{shen-yu-zhang-errata}, where Shen, Yu and Zhang state that their proof works with slight modifications; it is not clear to the author why the argument offered in the errata is sufficient to complete the proof.

    \item
    Zink \cite{zink-02} introduces the notion of a display that is a non-perfect version of the notion of a Dieudonn\'e module.
    B\"ultel and Pappas \cite{bueltel-pappas} define a group-theoretic version; this gives a deperfection of the notion of a local shtuka from \cite{xiao-zhu} in the hyperspecial situation.
    In a parahoric Hodge-type situation, Pappas \cite{pappas-parahoric-disps} also gives a definition of group-theoretic displays, however only over $p$-torsion-free $p$-adic rings.
    Pappas also constructs a group-theoretic display on the $p$-completion of a Kisin--Pappas integral model.

\end{itemize}

\subsection{Overview}

Let us explain the content of the present work.
The letter $R$ denotes a $p$-nilpotent ring.

We start by recalling the definition of a (3n-)display in the sense of \cite{zink-02}.

\begin{definition}[\textit{cf.} Definition~\ref{def:disps}] \label{intro-def:disps}
    Let $0 \leq d \leq h$ be integers.
    Then a \emph{display of type $(h, d)$ over $R$} is a tuple $(M, M_1, \Psi)$ that is given as follows: 
    \begin{itemize}
        \item
        $M$ is a finite projective $W(R)$-module of rank $h$.

        \item
        $M_1 \subseteq M$ is a $W(R)$-submodule containing $I_R M$ and such that $M_1/I_R M \subseteq M/I_R M$ is a direct summand of rank $d$.

        \item
        $\Psi \colon \widetilde{M_1} \to M$ is an isomorphism of $W(R)$-modules that we call the \emph{divided Frobenius}.
    \end{itemize}
    Here $W(R)$ denotes the ring of Witt vectors of $R$, $I_R$ denotes the kernel of the projection $W(R) \to R$, and the object $\widetilde{M_1}$ is a certain finite projective $W(R)$-module attached to $(M, M_1)$; see \Cref{def:m-1-tilde} for more details.
    If $R$ is a perfect ring of characteristic $p$, then $\widetilde{M_1}$ agrees with the Frobenius twist $M_1^{\sigma}$.

    The category of displays carries a natural duality and an action of the symmetric monoidal category of tuples $(I, \iota)$ consisting of an invertible $W(R)$-module $I$ and an isomorphism $\iota \colon I^{\sigma} \to I$.
\end{definition}

For positive integers $m$ and $n$ with $m \geq n + 1$, we define the notion of an $(m, n)$-truncated display that is inspired by the restricted local shtukas from \cite[Definition 5.3.1]{xiao-zhu}.
The word \enquote{truncated} refers to the use of truncated Witt vectors; the numbers $m$ and $n$ roughly measure how truncated the module $M$ and the divided Frobenius $\Psi$ are.
One should note that this notion of an $(m, n)$-truncated display is different from the notion of an $n$-truncated display as defined by Lau and Zink in \cite{lau-zink-18}; see \Cref{rmk:relation-lau-zink} for a comparison.
For us it will be crucial to work with the $(m, n)$-truncated objects; see \Cref{rmk:chains-lau-zink}.

The following theorem gives a relation between displays and $p$-divisible groups.

\begin{theorem}[\textit{cf.} \protect{\cite[Theorem 5.1]{lau-10}}, \Cref{thm:disp-p-div}] \label{intro-thm:disp-p-div}
    There is a natural functor
    \[
        \DD \colon \curlybr*{\text{$p$-divisible groups of height $h$ and dimension $d$ over $R$}}^{\op} \lra \curlybr*{\text{displays of type $(h, d)$ over $R$}}.
    \]
   that restricts to an equivalence between formal $p$-divisible groups and $F$-nilpotent displays.
\end{theorem}

As we are interested in studying the moduli space $\calA_{g, J, N}$ of polarized chains of Abelian varieties, it makes sense to make the following definition.

\begin{definition}[\textit{cf.}  \Cref{def:polarized-chains}] \label{intro-def:chains-disps}
    A \emph{homogeneously polarized chain of displays of type $(g, J)$ over $R$} is a tuple
    \[
        \roundbr*{(M_i)_i, \left(\rho_{i, j}\right)_{i, j}, (\theta_i)_i, (M_{i, 1})_i, (\Psi_i)_i, I, \iota, (\lambda_i)_i}
    \]
    that is given as follows: 
    \begin{itemize}
        \item
        $((M_i, M_{i, 1}, \Psi_i)_i, (\rho_{i, j})_{i, j})$ is a diagram of shape $J$ in the category of displays of type $(2g, J)$ such that the homomorphism of $R$-modules
        $
            \rho_{i, j} \colon R \otimes_{W(R)} M_i \to R \otimes_{W(R)} M_j
            $
        is of constant rank $2g - (j - i)$ for all $i \leq j \leq i + 2g$.

        \item
        $\theta_i \colon (M_i, M_{i, 1}, \Psi_i) \to (M_{i + 2g}, M_{i + 2g, 1}, \Psi_{i + 2g})$ is an isomorphism such that we have the compatibilities $\theta_j \circ \rho_{i, j} = \rho_{i + 2g, j + 2g} \circ \theta_i$ and $\rho_{i, i + 2g} = p \theta_i$.

        \item
        $(I, \iota)$ is as in \Cref{intro-def:disps}.

        \item
        $\lambda_i \colon (M_i, M_{i, 1}, \Psi_i) \to (I, \iota) \otimes (M_{-i}, M_{-i, 1}, \Psi_{-i})^{\vee}$ is an antisymmetric isomorphism such that we have $\lambda_j \circ \rho_{i, j} = (\id_{(I, \iota)} \otimes \rho_{-j, -i}^{\vee}) \circ \lambda_i$.
    \end{itemize}
\end{definition}

We again also give an $(m, n)$-truncated version of this definition.
If $R$ is of characteristic $p$, then we allow $n$ to take the additional value $1 \blank \rdt$ that can be thought of as being slightly smaller than $1$; \enquote{$\rdt$} refers to the term \enquote{reductive quotient}.
Roughly, the case $n = 1 \blank \rdt$ corresponds to only having a divided Frobenius on the graded pieces of the $1$-truncated chain of modules; see \Cref{def:chains}.

We then show that the stack $\cathpolchdisps_{g, J}^{(m, n)}$ of $(m, n)$-truncated homogeneously polarized chains of displays over $\Spf(\Zp)$ admits a quotient stack description.

\begin{proposition}[\textit{cf.} \Cref{lem:polarized-disps-quotient-stack}] \label{intro-lem:hpolchainsdisps}
    There exists an equivalence
    \[
        \cathpolchdisps^{(m, n)}_{g, J} \lra \squarebr*{\roundbr*{L^{(m)} \calG}_{\Delta} \backslash \rmM^{\loc, (n)}},
    \]
    where we use the following notation: 
    \begin{itemize}
        \item
        $L^{(m)} \calG$ denotes the $m$-truncated Witt vector positive loop group of $\calG$; see \Cref{subsec:witt-vectors}. 

        \item
        $\rmM^{\loc}$ is the $p$-completion of the local model $\bbM^{\loc}$,  and $\rmM^{\loc, (n)} \to \rmM^{\loc}$ is a certain $L^{(m)} \calG$-equivariant $L^{(n)} \calG$-torsor;  see \Cref{def:m-loc-sp-+}. 

        \item
        The subscript $\Delta$ indicates that we take the quotient by the diagonal action.
    \end{itemize}
\end{proposition}
Thus our definition of $\cathpolchdisps^{(m, n)}_{g, J}$ gives (up to a slight difference in normalization) a deperfection of the stack of parahoric restricted local shtukas from \cite[Section 4]{shen-yu-zhang}; 
see \Cref{rmk:relation-shtukas-2}.
In particular, we obtain bijections
\[
    \cathpolchdisps_{g, J} \roundbr*{\Fpbar} \lra \breve{K}_{\sigma} \backslash X
    \quad \text{and} \quad
    \cathpolchdisps_{g, J}^{(m, 1 \blank \rdt)} \roundbr*{\Fpbar} \lra \breve{K}_{\sigma} \backslash \left(\breve{K}_1 \backslash X\right).
\]

Applying \Cref{intro-thm:disp-p-div} to the moduli description of $\calA_{g, J, N}$, we thus obtain a natural morphism
\[
    \Upsilon \colon \calA_{g, J, N}^{\wedge} \lra \cathpolchdisps_{g, J}
\]
that realizes the central leaves map;  see \Cref{def:upsilon}. Here $\calA_{g, J, N}^{\wedge}$ denotes the $p$-completion of $\calA_{g, J, N}$.
For any $m \geq 2$ the composition
\[
    \roundbr*{\calA_{g, J, N}}_{\Fp} \stackrel{\Upsilon}\lra \cathpolchdisps_{g, J, \Fp} \lra \cathpolchdisps_{g, J}^{(m, 1 \blank \rdt)}
\]
then realizes the map $\upsilon$ parametrizing the EKOR stratification.
Our main result is now the following.

\begin{theorem}[\textit{cf.}~\Cref{thm:main-res}] \label{intro-thm:main-res}
  For every tuple of integers
  $(m, n)$ with $n \neq 1 \blank \rdt$, the natural morphism $\calA_{g, J, N}^{\wedge} \to \cathpolchdisps_{g, J}^{(m, n)}$ is smooth.
    Similarly, for every $m \geq 2$  the morphism $(\calA_{g, J, N})_{\Fp} \to \cathpolchdisps_{g, J}^{(m, 1 \blank \rdt)}$ is smooth as well.
\end{theorem}

\begin{proof}[Strategy of proof]
    By the Serre--Tate theorem, the smoothness of the morphism at a point of $\abs{\calA_{g, J, N}^{\wedge}}$ corresponding to a polarized chain of Abelian varieties only depends on the associated polarized chain of $p$-divisible groups.
    Using \Cref{intro-thm:disp-p-div} we then show that the morphism is smooth along the formal locus, \textit{i.e.}~the locus of chains of Abelian varieties with formal $p$-divisible groups.
    To obtain the smoothness in general, we then finally show that there are enough points in $\abs{\calA_{g, J, N}^{\wedge}}$ that specialize into the formal locus; see \Cref{cor:specializing} for a precise statement.
\end{proof}

As a natural next step, it would be interesting to generalize our results to the case of a more general Shimura variety at parahoric level instead of the Siegel modular variety.
It could be expected that there exists a stack of $(\calG, \mu)$-displays for every datum $(\calG, \mu)$ consisting of a parahoric $\Zp$-group scheme $\calG$ and a minuscule geometric conjugacy class $\mu$ of cocharacters of the generic fiber of $\calG$, as well as truncated versions that recover the definition from \cite{bueltel-pappas} in the hyperspecial case and give back our definition when the generic fiber of $\calG$ is either a general linear group or a group of symplectic similitudes.
In a situation where the local group datum $(\calG, \mu)$ comes from a Shimura datum, there then should be a natural smooth morphism from the $p$-completion of the corresponding Shimura variety into the stack of truncated $(\calG, \mu)$-displays.
As already mentioned above, partial results in this direction have been achieved by Pappas in \cite{pappas-parahoric-disps}.

\subsection{Acknowledgements}

I am very grateful to Ulrich G\"ortz for introducing me to the subject of Shimura varieties and their special fibers and for all the support I received from him throughout this project.
Furthermore, I would like to thank Jochen Heinloth, Ludvig Modin, Herman Rohrbach, Pol van Hoften and Torsten Wedhorn for helpful conversations.
Finally I also want to thank the anonymous referee for their valuable feedback, suggesting significant improvements to an earlier version of the manuscript.

    \section{Pairs and displays} \label{sec:pairs-disps}

In this section we  recall some of the theory of (not necessarily nilpotent) displays from \cite{zink-02}.
We also develop an analogous theory of $(m, n)$-truncated displays that is inspired by the definition of $(m, n)$-restricted local shtukas given in \cite{xiao-zhu}.

The letter $R$ denotes a $p$-nilpotent ring, and $(m, n)$ denotes a tuple of positive integers with $m \geq n + 1$.

\subsection{Witt vectors} \label{subsec:witt-vectors}

We use the following notation concerning Witt vectors: 

\begin{itemize}
    \item
    We write $W(R)$ for the ring of \emph{$(p$-typical\,$)$ Witt vectors of $R$} and $I_R \subseteq W(R)$ for its \emph{augmentation ideal}, \textit{i.e.}~the kernel of the projection $W(R) \to R$.

    Similarly, we write $W_n(R)$ for the ring of $n$-truncated Witt vectors of $R$ and $I_{n, R} \subseteq W_n(R)$ for its augmentation ideal.

    \item
    We denote the \emph{Witt vector Frobenius} on $W(R)$ by $\sigma \colon W(R) \to W(R)$.
    Given a $W(R)$-module $M$, we write $M^{\sigma} \coloneqq W(R) \otimes_{\sigma, W(R)} M$ for its Frobenius twist.

    Note that $\sigma$ induces a ring homomorphism $\sigma \colon W_m(R) \to W_n(R)$ on truncated Witt vectors (here it is crucial that $m \geq n + 1$, at least when $R$ is not of characteristic $p$).

    When $R$ is of characteristic $p$, we also write $\sigma \colon R \to R$ for the $p$-power Frobenius on $R$.

    \item
    We write
    \[
        W(R)^{\sigma = \id} \coloneqq \left\{x \in W(R)\mid\sigma(x) = x\right\} \subseteq W(R)
    \]
    for the subring of $\sigma$-invariant elements.
    Note that we have a natural isomorphism
    \[
        \Zp(R) \coloneqq \Cont \roundbr*{\abs{\Spec(R)}, \Zp} \lra W(R)^{\sigma = \id};
    \]
    when $R$ is of characteristic $p$, this follows from the chain of identifications
    \[
        \Zp(R) \cong W \roundbr*{\Fp(R)} \cong W \roundbr*{R^{\sigma = \id}} \cong W(R)^{\sigma = \id}, 
    \]
    and the general case follows because the reduction morphism $\Zp(R) \to \Zp(R/pR)$ is an isomorphism and the $\sigma$-stable ideal $\ker(W(R) \to W(R/pR)) \subseteq W(R)$ is killed by a power of $\sigma$, so  every $\sigma$-invariant element $x \in W(R/pR)$ lifts uniquely to a $\sigma$-invariant element $\widetilde{x} \in W(R)$.

    Similarly, we write
    \[
        W_m(R)^{\sigma = \id} \coloneqq \left\{x \in W_m(R)\mid\sigma(x) = x \in W_n(R)\right\} \subseteq W_m(R).
    \]
    Note that this is slightly ambiguous as the chosen $n$ does not appear in the notation; however, it should always be clear from the context what $n$ is used.

    \item
    Recall that the inverse of the \emph{Verschiebung} is a $\sigma$-linear map $I_R \to W(R)$.
    We denote its linearization by $\sigma^{\divd} \colon I_R^{\sigma} \to W(R)$ and call it the \emph{divided Frobenius}.
    For $x \in I_R$ we have $p \cdot \sigma^{\divd}(1 \otimes x) = \sigma(x)$, justifying the name.

    We also have a truncated variant of the divided Frobenius $\sigma^{\divd} \colon W_n(R) \otimes_{\sigma, W_m(R)} I_{m, R} \to W_n(R)$.
    
    \item
    Let $M$, $N$ be finite projective $W(R)$-modules, and let $f \colon M \to I_R N \subseteq N$ be a $W(R)$-linear map.
    We write $f^{\sigma, \divd}$ for the composition
    \[
        f^{\sigma, \divd} \colon M^{\sigma} \xrightarrow{f^{\sigma}} I_R^{\sigma} \otimes_{W(R)} N^{\sigma} \xrightarrow{\sigma^{\divd} \otimes \id} N^{\sigma}.
    \]
    Then we have $p \cdot f^{\sigma, \divd} = f^{\sigma}$, and given homomorphisms of finite projective $W(R)$-modules $g \colon L \to M$ and $h \colon N \to P$, we have $(f \circ g)^{\sigma, \divd} = f^{\sigma, \divd} \circ g^{\sigma}$ and $(h \circ f)^{\sigma, \divd} = h^{\sigma} \circ f^{\sigma, \divd}$.

    Similarly, given finite projective $W_m(R)$-modules $M$ and $N$ and a homomorphism of $W_m(R)$-modules $f \colon M \to I_{m, R} N \subseteq N$, we write $f^{\sigma, \divd}$ for the composition
    \begin{align*}
        f^{\sigma, \divd} \colon W_n(R) \otimes_{\sigma, W_m(R)} M &\xrightarrow{f^{\sigma}} \roundbr*{W_n(R) \otimes_{\sigma, W_m(R)} I_{m, R}} \otimes_{W_n(R)} \roundbr*{W_n(R) \otimes_{\sigma, W_m(R)} N} \\
        & \xrightarrow{\sigma^{\divd} \otimes \id} W_n(R) \otimes_{\sigma, W_m(R)} N.
    \end{align*}
    This construction has similar properties to the non-truncated version.

    \item
    Given a smooth affine $\Zp$-group scheme $G$, we write $L^+ G$ for the \emph{$($Witt vector$)$ positive loop group of $G$}, \textit{i.e.}~the flat affine $\Zp$-group scheme given by $(L^+ G)(R) = G(W(R))$ (see also \cite[Section 2.2]{bueltel-hedayatzadeh}).

    We also write $L^{(n)} G$ for the $n$-truncated positive loop group of $G$, \textit{i.e.}~the smooth affine $\Zp$-group scheme given by $(L^{(n)} G)(R) = G(W_n(R))$.
\end{itemize}

We also record the following technical lemma that will be used multiple times throughout the article.

\begin{lemma} \label{lem:chains-rank}
    Suppose that we are given an admissible linearly topologized ring $A$ $($see \cite[Tag 07E8]{stacks-project}$)$ such that $p$ is topologically nilpotent in $A$.
    Let $M$, $M'$ be finite projective $A$-modules of rank $h$, and let $f \colon M \to M'$ and $g \colon M' \to M$ be morphisms of $A$-modules such that $g \circ f = p \cdot \id_M$ and $f \circ g = p \cdot \id_{M'}$.
    Suppose furthermore that there exist integers $\ell$ and $\ell'$ with $\ell + \ell' = h$ such that for every continuous ring homomorphism $A \to k$ with $k$ an algebraically closed field, the induced homomorphisms of\, $k$-vector spaces
    \[
        f \colon k \otimes_A M \lra k \otimes_A M' \quad \text{and} \quad g \colon k \otimes_A M' \lra k \otimes_A M
    \]
    are of ranks $\ell$ and $\ell'$.

    Then the induced homomorphisms of $A/pA$-modules
    \[
        f \colon M/pM \lra M'/pM' \quad \text{and} \quad g \colon M'/pM' \lra M/pM
    \]
    are of constant ranks $\ell$ and $\ell'$, in the sense that their respective images are direct summands $($hence finite projective$)$ of the indicated rank.
\end{lemma}

\begin{proof}
    Let $\fraka \subseteq A$ be an ideal of definition that contains $p$.
    Note that the condition that $f$ and $g$ are of constant ranks $\ell$ and $\ell'$, is representable by a finitely presented locally closed subscheme $Z \subseteq \Spec(A/\fraka)$.
    By assumption we have $\abs{Z} = \abs{\Spec(A/\fraka)}$, so  $Z$ is the vanishing locus of a nilpotent ideal $\frakb \subseteq A/\fraka$.
    After replacing $\fraka$ by the preimage of $\frakb$ under $A \to A/\fraka$,  we thus may assume that the homomorphisms of $A/\fraka$-modules
    \[
        f \colon M/\fraka M \lra M'/\fraka M' \quad \text{and} \quad g \colon M'/\fraka M' \lra M/\fraka M
    \]
    are of constant ranks $\ell$ and $\ell'$.

    Modulo $\fraka$ the morphisms $f$ and $g$ now compose to $0$ in both directions and are of complementary constant ranks.
    Thus there exist direct sum decompositions
    \[
        M/\fraka M = P \oplus P' \quad \text{and} \quad M'/\fraka M' = P \oplus P'
    \]
    with respect to which we have $f = \begin{psmallmatrix} \id_P & 0 \\ 0 & 0 \end{psmallmatrix}$ and $g = \begin{psmallmatrix} 0 & 0 \\ 0 & \id_{P'} \end{psmallmatrix}$.
    Lift these decompositions to
    \[
        M = Q \oplus Q' \quad \text{and} \quad M' = Q \oplus Q'.
    \]
    After possibly modifying the lifts, we may then assume that
    \[
        f = \begin{pmatrix} \id_Q & 0 \\ 0 & \ast \end{pmatrix} \quad \text{and} \quad g = \begin{pmatrix} \ast & \ast \\ \ast & \id_{Q'} \end{pmatrix}.
    \]
    The assumptions $g \circ f = p \cdot \id_M$ and $f \circ g = p \cdot \id_{M'}$ then imply that we in fact have
    \[
        f = \begin{pmatrix} \id_Q & 0 \\ 0 & p \cdot \id_{Q'} \end{pmatrix} \quad \text{and} \quad g = \begin{pmatrix} p \cdot \id_Q & 0 \\ 0 & \id_{Q'} \end{pmatrix}, 
    \]
    so the claim follows.
\end{proof}

\subsection{Pairs and displays} \label{subsec:pairs-displays}

Let $h$ and $d$ denote integers with $0 \leq d \leq h$.

\begin{definition}
    A \emph{pair $($of type $(h, d))$ over $R$} is a tuple $(M, M_1)$ consisting of a finite projective $W(R)$-module $M$ (of rank $h$) and a $W(R)$-submodule $M_1 \subseteq M$ with $I_R M \subseteq M_1$ and such that $M_1/I_R M \subseteq M/I_R M$ is a direct summand (of rank $d$).

    An \emph{$m$-truncated pair over $R$} is a tuple $(M, M_1)$ consisting of a finite projective $W_m(R)$-module $M$ and a $W_m(R)$-submodule $M_1 \subseteq M$ with $I_{m, R} M \subseteq M_1$ and such that $M_1/I_{m, R} M \subseteq M/I_{m, R} M$ is a direct summand.
\end{definition}

\begin{definition}
    Let $(M, M_1)$ and $(M', M'_1)$ be two pairs over $R$.
    Then a \emph{morphism} $f \colon (M, M_1) \to (M', M'_1)$ is a morphism of $W(R)$-modules $f \colon M \to M'$ such that $f(M_1) \subseteq M'_1$.

    In the same way, we also define morphisms of $m$-truncated pairs.
\end{definition}

\begin{remark} \label{rmk:base-changing-pairs}
    Let $(M, M_1)$ be a pair over $R$, and let $R \to R'$ be a morphism of $p$-complete rings.
    Then we can form the base change $(M', M'_1) = (M, M_1)_{R'}$ that is a pair over $R'$.
    It is characterized by
    \[
        M' = W(R') \otimes_{W(R)} M \quad \text{and} \quad M'_1/I_{R'} M' = R' \otimes_R \left(M_1/I_R M\right).
    \]
    Similarly, we can also base change $m$-truncated pairs.

    From the descent result \cite[Corollary 34]{zink-02}, it follows that the assignment $R \mapsto \curlybr{\text{pairs over $R$}}$ defines a stack of $W(\calO_{\Spf(\Zp)})$-linear categories for the fpqc topology; here $W(\calO_{\Spf(\Zp)})$ denotes the sheaf of rings that is given by $R \mapsto W(R)$.

    Similarly, using \cite[Lemma 3.12]{lau-10}, we see that the assignment $R \mapsto \curlybr{\text{$m$-truncated pairs over $R$}}$ defines a stack of $W_m(\calO_{\Spf(\Zp)})$-linear categories.
\end{remark}

\begin{remark} \label{rmk:truncating-pairs}
    There are natural truncation functors
    \[
        \curlybr*{\text{pairs over $R$}} \lra \curlybr*{\text{$m$-truncated pairs over $R$}}
    \]
    and
    \[
        \curlybr*{\text{$m'$-truncated pairs over $R$}} \lra \curlybr*{\text{$m$-truncated pairs over $R$}}
    \]
    for $m \leq m'$.
\end{remark}

\begin{lemma}
    Let $(M, M_1)$ be a pair over $R$.
    Then $(M, M_1)$ has a \emph{normal decomposition} $(L, T)$, i.e.\ a direct sum decomposition $M = L \oplus T$ such that $M_1 = L \oplus I_R T$.
    Given a second pair $(M', M'_1)$ with normal decomposition $(L', T')$, every morphism of pairs $f \colon (M, M_1) \to (M', M'_1)$ can be written in matrix form $f = \begin{psmallmatrix} a & b \\ c & d \end{psmallmatrix}$ with
    \[
        a \colon L \lra L', \quad b \colon T \lra L', \quad c \colon L \lra I_R T', \quad d \colon T \lra T'.
    \]

    The same is true for $m$-truncated pairs.
\end{lemma}

\begin{proof}
    Set $L' \coloneqq M_1/I_R M \subseteq M/I_R M$.
    By definition it is a direct summand, so  we can choose a complement $T' \subseteq M/I_R M$.
    As $W(R)$ is Henselian along $I_R$,  we can now lift the decomposition $M/I_R M = L' \oplus T'$ to a decomposition $M = L \oplus T$ as desired.
    The claim about writing morphisms as matrices with respect to chosen normal decompositions is immediate.
\end{proof}

\begin{definition} \label{def:m-1-tilde}
    We define a natural $\sigma$-linear functor
    \[
    \begin{array}{r c l}
        \curlybr*{\text{pairs (of type $(h, d)$) over $R$}}
        & \lra
        & \curlybr*{\begin{gathered}\text{finite projective $W(R)$-modules} \\ \text{(of rank $h$)}\end{gathered}}, \\[5mm]
        (M, M_1)
        & \longmapsto
        & \widetilde{M_1}
    \end{array}
    \]
    as follows: 
    \begin{itemize}
        \item
        Given a pair $(M, M_1)$ over $R$ with normal decomposition $(L, T)$, we set $\widetilde{M_1} \coloneqq L^{\sigma} \oplus T^{\sigma}$.

        \item
        Given two pairs $(M, M_1)$ and $(M', M'_1)$ with normal decompositions $(L, T)$ and $(L', T')$ and a morphism $f = \begin{psmallmatrix} a & b \\ c & d \end{psmallmatrix} \colon (M, M_1) \to (M', M'_1)$, we define
        \[ 
            \widetilde{f} \coloneqq \begin{pmatrix} a^{\sigma} & p \cdot b^{\sigma} \\ c^{\sigma, \divd} & d^{\sigma} \end{pmatrix} \colon \widetilde{M_1} \lra \widetilde{M'_1}.
        \]
    \end{itemize}

    In the same way we also define a natural functor
    \[
    \begin{array}{r c l}
        \curlybr*{\text{$m$-truncated pairs over $R$}}
        & \lra
        & \curlybr*{\text{finite projective $W_n(R)$-modules}}, \\[1mm]
        (M, M_1)
        & \longmapsto
        & \widetilde{M_1}.
    \end{array}
    \]
\end{definition}

\begin{remark} \label{rmk:m-1-tilde}
    We make the following remarks.
    \begin{itemize}
        \item
        Checking that the functor $(M, M_1) \mapsto \widetilde{M_1}$ from \Cref{def:m-1-tilde} is well defined amounts to checking that the definition of $\widetilde{f}$ is compatible with identities and composition.
        This can be easily verified using the properties of $(\blank)^{\sigma, \divd}$; see \Cref{subsec:witt-vectors}.

        The definition $\widetilde{f} \coloneqq \begin{psmallmatrix} a^{\sigma} & p \cdot b^{\sigma} \\ c^{\sigma, \divd} & d^{\sigma} \end{psmallmatrix}$ imitates the expression $\begin{psmallmatrix} 1 & 0 \\ 0 & p \end{psmallmatrix}^{-1} \begin{psmallmatrix} a & b \\ c & d \end{psmallmatrix}^{\sigma} \begin{psmallmatrix} 1 & 0 \\ 0 & p \end{psmallmatrix}$; the latter \textit{a priori} only makes sense after base changing to $W(R)[1/p]$, but then the two terms agree.

        \item
        Let $(M, M_1)$ be a pair over $R$.
        Then we can informally think of $\widetilde{M_1}$ as the \enquote{correct version} of the Frobenius twist $M_1^{\sigma}$ that usually fails to be a finite projective $W(R)$-module.

        We have a natural surjective $W(R)$-linear map
        \[
            M_1^{\sigma} \cong L^{\sigma} \oplus \roundbr*{I_R^{\sigma} \otimes_{W(R)} T^{\sigma}} \xrightarrow{(\id, \sigma^{\divd} \otimes \id)} L^{\sigma} \oplus T^{\sigma} \cong \widetilde{M_1}, 
        \]
        and this map actually is an isomorphism when $R$ is a perfect ring of characteristic $p$.

        We also have natural $W(R)$-linear maps
        \[
            \widetilde{M_1} \cong L^{\sigma} \oplus T^{\sigma} \xrightarrow{\begin{psmallmatrix} 1 & 0 \\ 0 & p \end{psmallmatrix}} L^{\sigma} \oplus T^{\sigma} \cong M^{\sigma}
            \quad \text{and} \quad
            M^{\sigma} \cong L^{\sigma} \oplus T^{\sigma} \xrightarrow{\begin{psmallmatrix} p & 0 \\ 0 & 1 \end{psmallmatrix}} L^{\sigma} \oplus T^{\sigma} \cong \widetilde{M_1}
        \]
        that we can informally think of as the \enquote{inclusion} and the \enquote{multiplication-by-$p$ map}, respectively.
        Similar maps also exist in the truncated situation.

        \item
        The theory of \emph{higher displays} as developed in \cite{lau-higher-frames} and \cite{daniels-tannakian-displays} gives a way to conceptualize \Cref{def:m-1-tilde}.

        Let $W(R)^{\oplus}$ be the $\ZZ$-graded ring underlying the Witt frame (see \cite[Definition 2.2]{daniels-tannakian-displays}), and recall that it is equipped with two ring homomorphisms $\tau, \sigma \colon W(R)^{\oplus} \to W(R)$.
        Now giving a pair $(M, M_1)$ over $R$ is equivalent to giving a finite projective graded $W(R)^{\oplus}$-module whose type is concentrated in $[0, 1]$, and the functors $(M, M_1) \mapsto M$ and $(M, M_1) \mapsto \widetilde{M_1}$ correspond to base changing along $\tau$ and $\sigma$, respectively.
    \end{itemize}
\end{remark}

\begin{definition} \label{def:disps}
    A \emph{display over $R$} is a tuple $(M, M_1, \Psi)$, where $(M, M_1)$ is a pair over $R$ and $\Psi \colon \widetilde{M_1} \to M$ is an isomorphism of $W(R)$-modules.
    We call $\Psi$ the \emph{divided Frobenius} of the display.

    Similarly, an \emph{$(m, n)$-truncated display over $R$} is a tuple $(M, M_1, \Psi)$, where $(M, M_1)$ is an $m$-truncated pair over $R$ and $\Psi \colon \widetilde{M_1} \to W_n(R) \otimes_{W_m(R)} M$ is an isomorphism of $W_n(R)$-modules.
\end{definition}

\begin{remark}
     Similarly to what was explained for pairs in \Cref{rmk:base-changing-pairs,rmk:truncating-pairs}, one can also base change and truncate displays.

    The assignment $R \mapsto \curlybr{\text{displays over $R$}}$ defines an fpqc-stack of $\underline{\Zp}$-linear categories; here $\underline{\Zp}$ denotes the sheaf that is given by $R \mapsto \Zp(R)$.

    Similarly, the assignment $R \mapsto \curlybr{\text{$(m, n)$-truncated displays over $R$}}$ defines a stack of $W_m(\calO_{\Spf(\Zp)})^{\sigma = \id}$-linear categories.
\end{remark}

\begin{definition} \label{def:frobenius}
    Let $(M, M_1, \Psi)$ be a display over $R$.
    We define the \emph{Frobenius of $(M, M_1, \Psi)$} as the composition
    \[
        F_M \colon M^{\sigma} \lra \widetilde{M_1} \stackrel{\Psi}\lra M,
    \]
    where the first arrow is the \enquote{multiplication-by-$p$ map} introduced in \Cref{rmk:m-1-tilde}.
    Furthermore, we say that $(M, M_1, \Psi)$ is \emph{$F$-nilpotent} if there exists some $N \in \ZZ_{\geq 0}$ such that the $R/pR$-module homomorphism
    \[
        R/pR \otimes \roundbr*{F_M \circ \dotsb \circ F_M^{\sigma^{N - 1}}} \colon R/pR \otimes_{W(R)} M^{\sigma^N} \lra R/pR \otimes_{W(R)} M
    \]
    vanishes.

    Now let $(M, M_1, \Psi)$ be an $(m, n)$-truncated display over $R$.
    Then we again define the \emph{Frobenius of $(M, M_1, \Psi)$} as the composition
    \[
        F_M \colon W_n(R) \otimes_{\sigma, W_m(R)} M \lra \widetilde{M_1} \stackrel{\Psi}\lra W_n(R) \otimes_{W_m(R)} M
    \]
    and say that $(M, M_1, \psi)$ is \emph{$F$-nilpotent} if there exists some $N \in \ZZ_{\geq 0}$ such that the $R/pR$-module homomorphism
    \[
        \roundbr*{R/pR \otimes_{W_n(R)} F_M} \circ \dotsb \circ \roundbr*{R/pR \otimes_{W_n(R)} F_M}^{(p^{N - 1})} \colon \roundbr*{R/pR \otimes_{W_m(R)} M}^{(p^N)} \lra R/pR \otimes_{W_m(R)} M
    \]
    vanishes.
\end{definition}

\begin{remark}
    It follows immediately from the definition that an ($(m, n)$-truncated) display over $R$ is $F$-nilpotent if and only if its $(2, 1)$-truncated base change to $R/pR$ is $F$-nilpotent.
\end{remark}

\begin{remark}[Relation with Dieudonn\'e modules] \label{rmk:relation-dieudonne}
    Suppose that $R$ is perfect of characteristic $p$.
    Then there is the classical notion of a \emph{Dieudonn\'e module over $R$}.

    Such a Dieudonn\'e module over $R$ is a tuple $(M, F_M)$ consisting of a finite projective $W(R)$-module $M$ and an isomorphism $F_M \colon M^{\sigma}[1/p] \to M[1/p]$ that satisfies $p M \subseteq F_M(M^{\sigma}) \subseteq M$.

    We have a natural equivalence of categories
    \[
        \curlybr*{\text{displays over $R$}} \lra \curlybr*{\text{Dieudonn\'e modules over $R$}}
    \]
    that sends a display $(M, M_1, \Psi)$ to $(M, F_M)$, where $F_M$ is the Frobenius from \Cref{def:frobenius}.
    The inverse of this equivalence is given by sending $(M, F_M)$ to $(M, M_1, \Psi)$, where
    \[
        M_1 \coloneqq p \cdot F_M^{-1}(M)^{\sigma^{-1}} \subseteq M
        \quad \text{and} \quad
        \Psi \colon \widetilde{M_1} \cong p \cdot F_M^{-1}(M) \xrightarrow{p^{-1} \cdot F_M} M.
    \]
    Moreover, the display corresponding to a Dieudonn\'e module $(M, F_M)$ over $R$ is of type $(h, d)$ if and only if $M$ is a finite projective $W(R)$-module of rank $h$ and $M/F_M(M^{\sigma})$ is a finite projective $R$-module of rank $d$.
\end{remark}

\begin{remark}[Relation with truncated displays in the sense of Lau and Zink] \label{rmk:relation-lau-zink}
    Let us recall the notions of \emph{truncated pairs and displays} from \cite{lau-zink-18}; see also \cite{lau-10}.
    
    Consider the ring
    \[
        \calW_n(R) \coloneqq W_{n + 1}(R)/ \roundbr*{0, \dotsc, 0, R[p]},
    \]
    where $R[p] \subseteq R$ denotes the $p$-torsion in $R$.
    We have $(0, \dotsc, 0, R[p]) \cdot I_{n + 1, R} = 0$, so  $I_{n + 1, R}$ is naturally a $\calW_n(R)$-module, and we have a Frobenius $\sigma \colon \calW_n(R) \to W_n(R)$ as well as a divided Frobenius $\sigma^{\divd} \colon W_n(R) \otimes_{\sigma, \calW_n(R)} I_{n + 1, R} \to W_n(R)$.

    We now have the $\calW_n(R)$-linear category of \emph{$n$-truncated Lau--Zink-pairs} that has the following description.
    Every $n$-truncated Lau--Zink-pair (that we informally denote by $(M, M_1)$) has a normal decomposition $(L, T)$, where $L$ and $T$ are finite projective $\calW_n(R)$-modules, and given two such objects $(M, M_1)$ and $(M', M'_1)$ with normal decompositions $(L, T)$ and $(L', T')$, morphisms $f \colon (M, M_1) \to (M', M'_1)$ are given by matrices $f = \begin{psmallmatrix} a & b \\ c & d \end{psmallmatrix}$ with
    \[
        a \colon L \lra L', \quad
        b \colon T \lra L', \quad
        c \colon L \lra I_{n + 1, R} \otimes_{\calW_n(R)} T', \quad
        d \colon T \lra T'.
    \]

    There again is a $\sigma$-linear functor
    \[
        \curlybr*{\text{$n$-truncated Lau--Zink-pairs over $R$}} \lra \curlybr*{\text{finite projective $W_n(R)$-modules}}, \quad (M, M_1) \longmapsto \widetilde{M_1}
    \]
    that sends an $n$-truncated Lau--Zink-pair $(M, M_1)$ with normal decomposition $(L, T)$ to
    \[
        \widetilde{M_1} \coloneqq \roundbr*{W_n(R) \otimes_{\sigma, \calW_n(R)} L} \oplus \roundbr*{W_n(R) \otimes_{\sigma, \calW_n(R)} T}
    \]
    and a morphism $f = \begin{psmallmatrix} a & b \\ c & d \end{psmallmatrix}$ as above to $\widetilde{f} \coloneqq \begin{psmallmatrix} a^{\sigma} & p \cdot b^{\sigma} \\ c^{\sigma, \divd} & d^{\sigma} \end{psmallmatrix}$ analogously to \Cref{def:m-1-tilde}.
    An \emph{$n$-truncated Lau--Zink-display} is then defined to be a tuple $(M, M_1, \Psi)$, where $(M, M_1)$ is an $n$-truncated Lau--Zink-pair and $\Psi \colon \widetilde{M_1} \to W_n(R) \otimes_{\calW_n(R)} M$ is an isomorphism.

    From this description it is clear that there exist natural truncation functors
    \[
        \curlybr*{\text{$(m', n')$-truncated displays over $R$}} \lra \curlybr*{\text{$n$-truncated Lau--Zink-displays over $R$}}
    \]
    for $n \leq n'$ and
    \[
        \curlybr*{\text{$n'$-truncated Lau--Zink-displays over $R$}} \lra \curlybr*{\text{$(m, n)$-truncated displays over $R$}}
    \]
    for $m \leq n'$.

    Let us remark that for $R$ of characteristic $p$, we have an equivalence
    \[
        \curlybr*{\text{$1$-truncated Lau--Zink-displays over $R$}} \lra \curlybr*{\text{$F$-zips over $R$}},
    \]
    where the right-hand side denotes the  category of $F$-Zips from \cite[Definition 1.5]{moonen-wedhorn} with type concentrated in $[0, 1]$; see \cite[Example 3.6.4]{lau-higher-frames}.
\end{remark}

\subsection{Duals and twists} \label{subsec:duals-twists}

\begin{definition}
    Let $(M, M_1)$ be a pair over $R$.
    Then we define its \emph{dual}
    \[
        (M, M_1)^{\vee} \coloneqq \left(M^{\vee}, M_1^{\ast}\right)
    \]
    as follows: 
    \begin{itemize}
        \item
        $M^{\vee} = \Hom_{W(R)}(M, W(R))$ is the dual of the finite projective $W(R)$-module $M$.

        \item
        $M_1^{\ast} \subseteq M$ is the $W(R)$-submodule of all $\omega \colon M \to W(R)$ such that $\omega(M_1) \subseteq I_R$.
        Equivalently, it is the preimage under
        \[
            M^{\vee} \lra M^{\vee}/I_R M^{\vee} \cong \left(M/I_R M\right)^{\vee}
        \]
        of the orthogonal complement $(M_1/I_R M)^{\perp} \subseteq (M/I_R M)^{\vee}$.
    \end{itemize}
    This endows the category of pairs over $R$ with a $W(R)$-linear duality in the sense of \Cref{def:cat-duality} and \Cref{rmk:cat-lambda}.

    We similarly define duals of $m$-truncated pairs.
\end{definition}

\begin{remark}
    Note that if $(M, M_1)$ is an ($m$-truncated) pair of type $(h, d)$ over $R$,  then its dual $(M, M_1)^{\vee}$ is of type $(h, h - d)$.
\end{remark}

\begin{lemma} \label{lem:m-1-tilde-dual}
    The functor
    \[
        \curlybr*{\,\text{pairs over $R$}} \lra \curlybr*{\,\text{finite projective $W(R)$-modules}\,}, \quad (M, M_1) \longmapsto \widetilde{M_1}
    \]
    is naturally compatible with dualities.

    The same is true for the functor
    \[
        \curlybr*{\text{$m$-truncated pairs over $R$}} \lra \curlybr*{\,\text{finite projective $W_n(R)$-modules}\,}, \quad (M, M_1) \longmapsto \widetilde{M_1}.
    \]
\end{lemma}

\begin{proof}
    Given a pair $(M, M_1)$ over $R$, we need to define a natural isomorphism $\widetilde{M_1^{\ast}} \to \widetilde{M_1}^{\vee}$.
    Let $(L, T)$ be a normal decomposition of $(M, M_1)$.
    Then $(T^{\vee}, L^{\vee})$ is a normal decomposition of $(M, M_1)^{\vee}$, and we can define the desired isomorphism as
    \[
        \begin{pmatrix} 0 & \id_{L^{\sigma, \vee}} \\ \id_{T^{\sigma, \vee}} & 0 \end{pmatrix} \colon \widetilde{M_1^{\ast}} \cong T^{\vee, \sigma} \oplus L^{\vee, \sigma} \lra \roundbr*{L^{\sigma} \oplus T^{\sigma}}^{\vee} \cong \widetilde{M_1}^{\vee}. \qedhere
    \]
\end{proof}

\begin{definition}
    Let $(M, M_1, \Psi)$ be a display over $R$.
    Then we define its \emph{dual}
    \[
        (M, M_1, \Psi)^{\vee} \coloneqq \left(M^{\vee}, M_1^{\ast}, \Psi^{\vee, -1}\right),
    \]
    where we implicitly use \Cref{lem:m-1-tilde-dual} to make sense of $\Psi^{\vee, -1}$ as an isomorphism $\widetilde{M_1^{\ast}} \to M^{\vee}$.
    This endows the category of displays over $R$ with a $\Zp(R)$-linear duality.

    We similarly define duals of $(m, n)$-truncated displays.
\end{definition}

\begin{definition}
    Let $(M, M_1)$ be a pair over $R$, and let $I$ be an finite projective $W(R)$-module.
    Then we define the \emph{twist}
    \[
        I \otimes (M, M_1) \coloneqq \left(I \otimes_{W(R)} M, I \otimes_{W(R)} M_1\right).
    \]
    This defines an action of the symmetric monoidal category of finite projective $W(R)$-modules on the category of pairs over $R$ in the sense of \Cref{def:cat-action}.

    We similarly define twists of $m$-truncated pairs over $R$ by finite projective $W_m(R)$-modules.
\end{definition}

\begin{remark}
    In fact the symmetric monoidal category of finite projective $W(R)$-modules is rigid and $W(R)$-linear, and the action defined above is naturally compatible with dualities and the $W(R)$-linear structure; see \Cref{rmk:cat-action-duality,rmk:cat-rigid-duality,rmk:cat-lambda}.

    In the same way,  the action in the truncated setting also is naturally compatible with dualities and the $W_m(R)$-linear structure.
\end{remark}

\begin{lemma} \label{lem:m-1-tilde-twist}
    The functor
    \[
        \curlybr*{\,\text{pairs over $R$}} \lra \curlybr*{\,\text{finite projective $W(R)$-modules}\,}, \quad (M, M_1) \longmapsto \widetilde{M_1}
    \]
    is naturally equivariant with respect to the symmetric monoidal functor
    \[
        \curlybr*{\,\text{finite projective $W(R)$-modules}\,} \lra \curlybr*{\,\text{finite projective $W(R)$-modules}\,}, \quad I \longmapsto I^{\sigma}
    \]
    in the sense of \Cref{def:cat-action} and \Cref{rmk:cat-action-duality}.

    Similarly, the functor
    \[
        \curlybr*{\text{$m$-truncated pairs over $R$}} \lra \curlybr*{\,\text{finite projective $W_n(R)$-modules}\,}, \quad (M, M_1) \longmapsto \widetilde{M_1}
    \]
    is naturally equivariant with respect to the symmetric monoidal functor
    \[
        \curlybr*{\,\text{finite projective $W_m(R)$-modules}\,} \lra \curlybr*{\,\text{finite projective $W_n(R)$-modules}\,}, \quad I \longmapsto W_n(R) \otimes_{\sigma, W_m(R)} I.
    \]
\end{lemma}

\begin{proof}
    Given a pair $(M, M_1)$ over $R$ and a finite projective $W(R)$-module $I$, we need to define a natural isomorphism $(I \otimes_{W(R)} M_1)^{\sim} \to I^{\sigma} \otimes_{W(R)} \widetilde{M_1}$.
    Let $(L, T)$ be a normal decomposition of $(M, M_1)$.
    Then $(I \otimes_{W(R)} L, I \otimes_{W(R)} T)$ is a normal decomposition of $I \otimes (M, M_1)$, and we can define the desired isomorphism as
    \[
        \left(I \otimes_{W(R)} M_1\right)^{\sim} \cong \left(I \otimes_{W(R)} L\right)^{\sigma} \oplus \left(I \otimes_{W(R)} T\right)^{\sigma} \lra I^{\sigma} \otimes_{W(R)} \left(L^{\sigma} \oplus T^{\sigma}\right) \cong I^{\sigma} \otimes_{W(R)} \widetilde{M_1}. \qedhere
    \]
\end{proof}

\begin{definition}
    Let $(M, M_1, \Psi)$ be a display over $R$, and let $(I, \iota)$ be a tuple consisting of a finite projective $W(R)$-module $I$ and an isomorphism $\iota \colon I^{\sigma} \to I$.
    Then we define the \emph{twist}
    \[
        (I, \iota) \otimes (M, M_1, \Psi) \coloneqq \left(I \otimes_{W(R)} M, I \otimes_{W(R)} M_1, \iota \otimes \Psi\right),
    \]
    where we implicitly use \Cref{lem:m-1-tilde-twist} to make sense of $\iota \otimes \Psi$ as an isomorphism $(I \otimes_{W(R)} M_1)^{\sim} \to M$.
    This defines an action of the symmetric monoidal category of tuples $(I, \iota)$ as above on the category of displays over $R$, and in fact this action is naturally compatible with dualities and the $\Zp(R)$-linear structure.

    We similarly define twists of $(m, n)$-truncated displays over $R$ by tuples $(I, \iota)$ consisting of a finite projective $W_m(R)$-module $I$ and an isomorphism $\iota \colon W_n(R) \otimes_{\sigma, W_m(R)} I \to W_n(R) \otimes_{W_m(R)} I$.
\end{definition}

\begin{remark}
    We can conceptually understand duals and twists of pairs and displays in terms of higher displays.

    Given a pair $(M, M_1)$ there is a corresponding finite projective graded module over $W(R)^{\oplus}$ with type concentrated in $[0, 1]$ (see \Cref{rmk:m-1-tilde}).
    The dual pair $(M, M_1)^{\vee}$ then corresponds to taking the dual of that module and shifting it by $1$, and the twist $I \otimes (M, M_1)$ corresponds to viewing $I$ as a finite projective graded $W(R)^{\oplus}$-module with type concentrated in $0$ and forming the tensor product over $W(R)^{\oplus}$.
    The various categorical properties of taking duals and twists then all follow from the fact that finite projective graded modules over a graded ring form a rigid symmetric monoidal category.
\end{remark}

\subsection{The display of a $\boldsymbol{p}$-divisible group} \label{subsec:disp-p-div}

We have the following result of Lau, relating $p$-divisible groups and displays.

\begin{theorem}[\textit{cf.} \cite{lau-10}] \label{thm:disp-p-div}
    There is a natural $\Zp(R)$-linear exact functor
    \[
        \DD \colon \curlybr*{\text{$p$-divisible groups over $R$}}^{\op} \lra \curlybr*{\text{displays over $R$}}
    \]
    that is compatible with dualities and satisfies the following properties: 
    \begin{itemize}
        \item
        For a $p$-divisible group $X$ over $R$ of height $h$ and dimension $d$,  the display $\DD(X)$ is of type $(h, d)$.

        \item
        Suppose that $R$ is a perfect ring of characteristic $p$.
        Then $\DD$ coincides with classical $($contravariant\,$)$ Dieudonn\'e theory.

        \item
        $\DD$ restricts to an equivalence
        \[
            \DD \colon \curlybr*{\,\text{formal $p$-divisible groups over $R$}}^{\op} \lra \curlybr*{\text{$F$-nilpotent displays over $R$}}.
        \]

        \item
        Let $X$, $X'$ be $p$-divisible groups over $R$ of height $h$, and let $f \colon X' \to X$ be an isogeny of height $r \in \ZZ_{\geq 0}$.
        Assume that there exists an isogeny $g \colon X \to X'$ such that $g \circ f = p \cdot \id_{X'}$ and $f \circ g = p \cdot \id_X$ $($note that this forces $r \leq h)$.
        Then the homomorphism of\, $W(R)/pW(R)$-modules
        \[
            \DD(f) \colon \DD(X)/p\DD(X) \lra \DD(X')/p\DD(X')
        \]
        is of constant rank $h - r$.
    \end{itemize}
\end{theorem}

\begin{proof}
    See \cite[Proposition 4.1 and Theorem 5.1]{lau-10}.
    For the last claim we can use \Cref{lem:chains-rank} to reduce to the situation where $R = k$ is an algebraically closed field of characteristic $p$, where the statement is a standard result from classical Dieudonn\'e theory.
\end{proof}     \section{(Polarized) chains of pairs and displays} \label{sec:chains-pairs-disps}

As in \Cref{sec:pairs-disps}, $R$ denotes a $p$-nilpotent ring and $(m, n)$ denotes a tuple of positive integers with $m \geq n + 1$.
When $R$ is of characteristic $p$,  we allow $n$ to take the additional value $1 \blank \rdt$ (where \enquote{$\rdt$} refers to the term \enquote{reductive quotient}) that we think of as being slightly smaller than $1$.
In the case $n = 1 \blank \rdt$, we require $m \geq 2$.

\subsection{Chains of pairs and displays} \label{subsec:chains-pairs-displays}

Fix a positive integer $h$, a second integer $0 \leq d \leq h$ and a non-empty subset $J \subseteq \ZZ$ such that $J + h \ZZ = J$.
Let $\calE$ be the set of \enquote{edges of $J$}, \textit{i.e.}~the set of tuples $(i, j)$ consisting of consecutive elements in $J$.
For $e = (i, j) \in \calE$ we write $\abs{e} \coloneqq j - i$ for the \enquote{length of $e$}.

\begin{definition} \label{def:chains}
    A \emph{chain $($of type $(h, J))$ over $R$} is a tuple
    \[
        \roundbr*{(M_i)_{i \in J}, \left(\rho_{i, j}\right)_{i, j \in J, i \leq j}, (\theta_i)_{i \in J}}
    \]
    that is given as follows: 
    \begin{itemize}
        \item
        $((M_i)_i, (\rho_{i, j})_{i, j})$ is a diagram of finite projective $W(R)$-modules of rank $h$ of shape $J$ such that the homomorphism of $R/pR$-modules
        \[
            \rho_{i, j} \colon R/pR \otimes_{W(R)} M_i \lra R/pR \otimes_{W(R)} M_j
        \]
        is of constant rank $h - (j - i)$ for all $i \leq j \leq i + h$.

        \item
        The $\theta_i \colon M_i \to M_{i + h}$ are isomorphisms such that we have $\theta_j \circ \rho_{i, j} = \rho_{i + h, j + h} \circ \theta_i$ and $\rho_{i, i + h} = p \cdot \theta_i$.
    \end{itemize}

    For $n \neq 1 \blank \rdt$ similarly define the notion of an \emph{$n$-truncated chain over $R$} by replacing $W(R)$ with $W_n(R)$ in the above.

    Now suppose that $R$ is of characteristic $p$.
    Then a \emph{$(1 \blank \rdt)$-truncated chain over $R$} is a tuple $((N_e)_{e \in \calE}, (\theta_e)_{e \in \calE})$ that is given as follows: 
    \begin{itemize}
        \item
        $N_e$ is a finite projective $R$-module of rank $\abs{e}$.

        \item
        $\theta_e \colon N_e \to N_{e + h}$ is an isomorphism.
    \end{itemize}
\end{definition}

\begin{remark} \label{rmk:chains-1-rdt}
    There are obvious natural truncation functors
    \[
        \curlybr*{\text{chains over $R$}} \lra \curlybr*{\text{$n$-truncated chains over $R$}}
    \]
    and
    \[
        \curlybr*{\text{$n'$-truncated chains over $R$}} \lra \curlybr*{\text{$n$-truncated chains over $R$}}
    \]
    for $n \leq n'$.

    When $R$ is of characteristic $p$,  we also define a truncation functor
    \[
        \curlybr*{\text{$1$-truncated chains over $R$}} \lra \curlybr*{\text{$(1 \blank \rdt)$-truncated chains over $R$}}
    \]
    by sending a $1$-truncated chain $((M_i)_i, (\rho_{i, j})_{i, j}, (\theta_i)_i)$ to the $(1 \blank \rdt)$-truncated chain $((N_e)_e, (\theta_e)_e)$ that is given as follows (where we write $e = (i, j)$): 
    \begin{itemize}
        \item
        $N_e \coloneqq \ker(\rho_{i, j}) \subseteq M_i$.

        \item
        $\theta_e \colon N_e \to N_{e + h}$ is the isomorphism induced by $\theta_i$.
    \end{itemize}
    Note that for $e = (i, j)$ as above we have an exact sequence
    \[
        M_i \xrightarrow{\rho_{i, j}} M_j \xrightarrow{\rho_{j, i + h}} M_{i + h} \xrightarrow{\rho_{i + h, j + h}} M_{j + h}, 
    \]
    so  we obtain an isomorphism $\ker(\rho_{i + h, j + h}) \cong \coker(\rho_{i, j})$.
    Composing this with $\theta_e$ yields an isomorphism $N_e \cong \coker(\rho_{i, j})$.
\end{remark}

\begin{remark}
    Let $((M_i)_i, (\rho_{i, j})_{i, j}, (\theta_i)_i)$ be a chain over $R$.
    Then actually already the morphism of $W(R)/p W(R)$-modules
    \[
        \rho_{i, j} \colon M_i/p M_i \lra M_j/p M_j
    \]
    is of constant rank $h - (j - i)$ for all $i \leq j \leq i + h$; this follows from \Cref{lem:chains-rank} applied to $A = W(R)$, $\rho = \rho_{i, j}$ and $\rho' = \theta_i^{-1} \circ \rho_{j, i + h}$.

    The same is true for $n$-truncated chains.
\end{remark}

\begin{definition}
    A \emph{chain of pairs $($of type $(h, J, d))$ over $R$} is a tuple
    \[
        \roundbr*{(M_i)_i, \left(\rho_{i, j}\right)_{i, j}, (\theta_i)_i, (M_{i, 1})_i}
    \]
    that is given as follows: 
    \begin{itemize}
        \item
        $((M_i)_i, (\rho_{i, j})_{i, j}, (\theta_i)_i)$ is a chain over $R$.

        \item
        $M_{i, 1} \subseteq M_i$ is a $W(R)$-submodule such that $(M_i, M_{i, 1})$ is a pair of type $(h, d)$ over $R$ and such that we have $\rho_{i, j}(M_{i, 1}) \subseteq M_{j, 1}$ and $\theta_i(M_{i, 1}) = M_{i + h, 1}$.
    \end{itemize}
    
    We similarly define the notion of an \emph{$m$-truncated chain of pairs over $R$}.
\end{definition}

\begin{proposition} \label{prop:m-1-tilde-chains}
    Let $((M_i)_i, (\rho_{i, j})_{i, j}, (\theta_i)_i, (M_{i, 1})_i)$ be a chain of pairs over $R$.
    Then the tuple
    \[
        \roundbr*{\left(\widetilde{M_{i, 1}}\right)_i, \left(\widetilde{\rho_{i, j}}\right)_{i, j}, \left(\widetilde{\theta_i}\right)_i}
    \]
    is a chain over $R$.

    The same is true for $m$-truncated chains of pairs $($where the construction yields an $n$-truncated chain$)$.
\end{proposition}

\begin{proof}
    We only need to check that the morphism of $R/pR$-modules
    \[
        \widetilde{\rho_{i, j}} \colon R/pR \otimes_{W(R)} \widetilde{M_{i, 1}} \lra R/pR \otimes_{W(R)} \widetilde{M_{j, 1}}
    \]
    is of constant rank $h - (j - i)$ for all $i \leq j \leq i + h$.
    \Cref{lem:chains-rank} allows us to reduce to the case where $R = k$ is an algebraically closed field of characteristic $p$.
    The morphism $\rho_{i, j}$ then gives rise to a commutative diagram
    \[
    \begin{tikzcd}
        M_{i, 1} \ar[r, "\rho_{i, j}"] \ar[d]
        & M_{j, 1} \ar[d]
        \\
        M_i \ar[r, "\rho_{i, j}"]
        & M_j
    \end{tikzcd}
    \]
    of free $W(k)$-modules of rank $h$ where all the morphisms are injective.
    Moreover, the cokernels of the vertical morphisms have length $h - d$, and the cokernel of the lower horizontal morphism has length $j - i$.
    Thus  the cokernel of the upper horizontal morphism also has length $j - i$.
    As the image of $M_{i, 1} \to M_{j, 1}$ contains $p M_{j, 1}$, we can conclude that the morphism of $k$-vector spaces $\rho_{i, j} \colon M_{i, 1}/pM_{i, 1} \to M_{j, 1}/pM_{j, 1}$ is of rank $h - (j - i)$.
    Twisting by $\sigma$ then gives the result (see \Cref{rmk:m-1-tilde}).

    The result for $m$-truncated chains of pairs now follows because every such $m$-truncated chain of pairs over $R$ can be lifted to a (non-truncated) chain of pairs over $R$.
\end{proof}

\begin{definition} \label{def:chains-disps}
    A \emph{chain of displays over $R$} is a tuple
    \[
        \roundbr*{(M_i)_i, \left(\rho_{i, j}\right)_{i, j}, (\theta_i)_i, (M_{i, 1})_i, (\Psi_i)_i}
    \]
    that is given as follows: 
    \begin{itemize}
        \item
        $((M_i)_i, (\rho_{i, j})_{i, j}, (\theta_i)_i, (M_{i, 1})_i)$ is a chain of pairs over $R$.

        \item
        \[
            (\Psi_i)_i \colon \roundbr*{\left(\widetilde{M_{i, 1}}\right)_i, \left(\widetilde{\rho_{i, j}}\right)_{i, j}, \left(\widetilde{\theta_i}\right)_i} \lra \roundbr*{(M_i)_i, \left(\rho_{i, j}\right)_{i, j}, (\theta_i)_i}
        \]
        is an isomorphism of chains over $R$.
    \end{itemize}
    
    We similarly define the notion of an \emph{$(m, n)$-truncated chain of displays over $R$}.

    We write
    \[
        \catchdisps_{h, J, d} \quad \text{and} \quad \catchdisps_{h, J, d}^{(m, n)}
    \]
    for the stacks over $\Spf(\Zp)$ of chains of displays and $(m, n)$-truncated chains of displays (the stacks $\catchdisps_{h, J, d}^{(m, 1 \blank \rdt)}$ are in fact defined over $\Spec(\Fp)$).
\end{definition}

\begin{remark} \label{rmk:chains-disps}
    Let $((M_i)_i, (\rho_{i, j})_{i, j}, (\theta_i)_i, (M_{i, 1})_i, (\Psi_i)_i)$ be a chain of displays over $R$.
    Then $(M_i, M_{i, 1}, \Psi_i)$ is a display of type $(h, d)$ over $R$, and we have morphisms of displays
    \[
        \rho_{i, j} \colon \left(M_i, M_{i, 1}, \Psi_i\right) \lra \left(M_j, M_{j, 1}, \Psi_j\right)
        \quad \text{and} \quad
        \theta_i \colon \left(M_i, M_{i, 1}, \Psi_i\right) \lra \left(M_{i + h}, M_{i + h, 1}, \Psi_{i + h}\right).
    \]
    Conversely, suppose we are given displays $(M_i, M_{i, 1}, \Psi_i)$ of type $(h, d)$ over $R$ and morphisms of displays $\rho_{i, j}$ and $\theta_i$ as above, and assume that $((M_i)_i, (\rho_{i, j})_{i, j}, (\theta_i)_i)$ is a chain over $R$.
    Then the tuple $((M_i)_i, (\rho_{i, j})_{i, j}, (\theta_i)_i, (M_{i, 1})_i, (\Psi_i)_i)$ is a chain of displays over $R$.

    The same is true for $(m, n)$-truncated chains of displays.
\end{remark}

\subsection{Polarized chains} \label{subsec:polarized-chains}

Let us now specialize to the situation where $h = 2g$ is even, $d = g$ and $-J = J$.
Given $e = (i, j) \in \calE$ we write $-e \coloneqq (-j, -i) \in \calE$.

\begin{remark}
    Recall that the categories of pairs and displays (over $R$) and their truncated variants carry a natural duality and a compatible action of a suitable rigid symmetric monoidal category; see \Cref{subsec:duals-twists}.

    Consequently, we obtain a similar structure on the categories of chains, chains of pairs, chains of displays and their truncated variants.
    The various base change, forgetful and truncation functors are naturally compatible with this extra structure.
    The situation is summarized by the following table: 

    \smallskip

    \begin{center}
    \begin{tabular}{c | c | c}
        \centering
        Category with duality
        & Symmetric monoidal category
        & Coefficients
        \\
        \hline
        chains (over $R$)
        & invertible $W(R)$-modules
        & $W(R)$
        \\
        chains of pairs
        & invertible $W(R)$-modules
        & $W(R)$
        \\
        chains of displays
        & \makecell{tuples $(I, \iota)$; $I$ an invertible $W(R)$-module \\ and $\iota \colon I^{\sigma} \to I$ an isomorphism}
        & $\Zp(R)$
        \\
        \hline
        $n$-truncated chains ($n \neq 1 \blank \rdt$)
        & invertible $W_n(R)$-modules
        & $W_n(R)$
        \\
        $1 \blank \rdt$-truncated chains
        & invertible $R$-modules
        & $R$
        \\
        $m$-truncated chains of pairs
        & invertible $W_m(R)$-modules
        & $W_m(R)$
        \\
        \makecell{$(m, n)$-truncated chains of displays \\ ($n \neq 1 \blank \rdt$)}
        & \makecell{tuples $(I, \iota)$; $I$ an invertible $W_m(R)$-module \\ and $\iota \colon W_n(R) \otimes_{\sigma, W_m(R)} I \to W_n(R) \otimes_{W_m(R)} I$ \\ an isomorphism}
        & $W_m(R)^{\sigma = \id}$
        \\
        $(m, 1 \blank \rdt)$-truncated chains of displays
        & \makecell{tuples $(I, \iota)$; $I$ an invertible $W_m(R)$-module \\ and $\iota \colon R \otimes_{\sigma, W_m(R)} I \to R \otimes_{W_m(R)} I$ \\ an isomorphism}
        & $W_m(R)^{\sigma = \id}$
        \\
    \end{tabular}
    \end{center}

    \bigskip

    For example, the dual of a chain $((M_i)_i, (\rho_{i, j})_{i, j}, (\theta_i)_i)$ over $R$ is
    \[
        \roundbr*{(M_i)_i, \left(\rho_{i, j}\right)_{i, j}, (\theta_i)_i}^{\vee} = \roundbr*{\left(M_{-i}^{\vee}\right)_i, \left(\rho_{-j, -i}^{\vee}\right)_{i, j}, \left(\theta_{-i-h}^{\vee}\right)_i}, 
    \]
    and its twist by an invertible $W(R)$-module $I$ is
    \[
        I \otimes \roundbr*{(M_i)_i, \left(\rho_{i, j}\right)_{i, j}, (\theta_i)_i} = \roundbr*{\left(I \otimes_{W(R)} M_i\right)_i, \left(\id_I \otimes \rho_{i, j}\right)_{i, j}, \left(\id_I \otimes \theta_i\right)_i}.
    \]
    Let us also spell out the duality coherence datum for the truncation functor
    \[
        \curlybr*{\text{$1$-truncated chains over $R$}} \lra \curlybr*{\text{$(1 \blank \rdt)$-truncated chains over $R$}}
    \]
    for $R$ of characteristic $p$.
    Let $((M_i)_i, (\rho_{i, j})_{i, j}, (\theta_i)_i)$ be a $1$-truncated chain over $R$, let $((N_e)_e, (\theta_e)_e)$ be its $(1 \blank \rdt)$-truncation, and let $((N'_e)_e, (\theta'_e)_e)$ be the $(1 \blank \rdt)$-truncation of its dual.
    Then we specify the natural isomorphism
    \[
        \roundbr*{\left(N'_e\right)_e, \left(\theta'_e\right)_e} \lra \roundbr*{\left(N_e\right)_e, \left(\theta_e\right)_e}^{\vee} = \roundbr*{\left(N_{-e}^{\vee}\right)_e, \left(\theta_{-e-h}^{\vee}\right)_e}
    \]
    to be given by
    \[
        N'_e \cong \ker\left(\rho_{-j, -i}^{\vee}\right) \lra \coker\left(\rho_{-j, -i}\right)^{\vee} \cong N_{-e}^{\vee}.
    \]
\end{remark}

\begin{definition} \label{def:polarized-chains}
    A \emph{polarized chain $($of type $(g, J))$ over $R$} is a polarized object in the category of chains (of type $(2g, J, g)$) over $R$, \textit{i.e.}~a tuple
    \[
        \roundbr*{(M_i)_i, \left(\rho_{i, j}\right)_{i, j}, (\theta_i)_i, (\lambda_i)_i}
    \]
    that is given as follows: 
    \begin{itemize}
        \item
        $((M_i)_i, (\rho_{i, j})_{i, j}, (\theta_i)_i)$ is a chain over $R$.

        \item
        $(\lambda_i)_i \colon ((M_i)_i, (\rho_{i, j})_{i, j}, (\theta_i)_i) \to ((M_i)_i, (\rho_{i, j})_{i, j}, (\theta_i)_i)^{\vee}$ is an antisymmetric isomorphism.
    \end{itemize}

    A \emph{homogeneously polarized chain over $R$} is a homogeneously polarized object in the category of chains over $R$, \textit{i.e.}~a tuple
    \[
        \roundbr*{(M_i)_i, \left(\rho_{i, j}\right)_{i, j}, (\theta_i)_i, I, (\lambda_i)_i}
    \]
    that is given as follows: 
    \begin{itemize}
        \item
        $((M_i)_i, (\rho_{i, j})_{i, j}, (\theta_i)_i)$ is a chain over $R$.

        \item
        $I$ is an invertible $W(R)$-module.

        \item
        $(\lambda_i)_i \colon ((M_i)_i, (\rho_{i, j})_{i, j}, (\theta_i)_i) \to I \otimes ((M_i)_i, (\rho_{i, j})_{i, j}, (\theta_i)_i)^{\vee}$ is an antisymmetric isomorphism.
    \end{itemize}

    In the same way we define the notions of \emph{$($homogeneously$)$ polarized chains of pairs and displays} as well as their truncated variants.

    We write
    \[
        \catpolchdisps_{g, J}, \quad
        \cathpolchdisps_{g, J}, \quad
        \catpolchdisps_{g, J}^{(m, n)}, \quad
        \cathpolchdisps_{g, J}^{(m, n)}
    \]
    for the stacks over $\Spf(\Zp)$ of polarized chains of displays, homogeneously polarized chains of displays and their truncated variants.
\end{definition}

\begin{remark} \label{rmk:m-1-tilde-polarized-chains}
    The functor
    \[
    \begin{array}{r c l}
        \curlybr*{\text{chains of pairs over $R$}}
        & \lra
        & \curlybr*{\text{chains over $R$}} 
        \\
        \roundbr*{(M_i)_i, \left(\rho_{i, j}\right)_{i, j}, (\theta_i)_i, (M_{i, 1})_i}
        & \longmapsto
        & \roundbr*{\left(\widetilde{M_{i, 1}}\right)_i, \left(\widetilde{\rho_{i, j}}\right)_{i, j}, \left(\widetilde{\theta_i}\right)_i}
    \end{array}
    \]
    (see \Cref{prop:m-1-tilde-chains}) is naturally compatible with dualities and equivariant with respect to the symmetric monoidal functor
    \[
        \curlybr*{\text{invertible $W(R)$-modules}} \lra \curlybr*{\text{invertible $W(R)$-modules}}, \quad I \longmapsto I^{\sigma};
    \]
    it thus induces functors
    \[
        \curlybr*{\text{(homogeneously) polarized chains of pairs over $R$}} \lra \curlybr*{\text{(homogeneously) polarized chains over $R$}}.
    \]
    A polarized chain of displays over $R$ can now be equivalently described as consisting of a polarized chain of pairs $((M_i)_i, (\rho_{i, j})_{i, j}, (\theta_i)_i, (M_{i, 1})_i, (\lambda_i)_i)$ over $R$ together with an isomorphism
    \[
        \left(\Psi_i\right)_i \colon \roundbr*{\left(\widetilde{M_{i, 1}}\right)_i, \left(\widetilde{\rho_{i, j}}\right)_{i, j}, \left(\widetilde{\theta_i}\right)_i, \left(\widetilde{\lambda_i}\right)_i} \lra \roundbr*{(M_i)_i, \left(\rho_{i, j}\right)_{i, j}, (\theta_i)_i, (\lambda_i)_i}, 
    \]
    and a homogeneously polarized chain of displays over $R$ can be equivalently described as consisting of a homogeneously polarized chain of pairs $((M_i)_i, (\rho_{i, j})_{i, j}, (\theta_i)_i, (M_{i, 1})_i, I, (\lambda_i)_i)$ over $R$ together with an isomorphism
    \[
        \roundbr*{\left(\Psi_i\right)_i, \iota} \colon \roundbr*{\left(\widetilde{M_{i, 1}}\right)_i, \left(\widetilde{\rho_{i, j}}\right)_{i, j}, \left(\widetilde{\theta_i}\right)_i, I^{\sigma}, (\widetilde{\lambda_i})_i} \lra \roundbr*{(M_i)_i, (\rho_{i, j})_{i, j}, (\theta_i)_i, I, (\lambda_i)_i}.
    \]

    The same is true for $(m, n)$-truncated (homogeneously) polarized chains of displays over $R$.
\end{remark}

\begin{lemma} \label{lem:restricting-pol-hpol}
    The morphism of stacks
    \[
        \catpolchdisps_{g, J} \lra \cathpolchdisps_{g, J}
    \]
    is a torsor under the flat affine $\Zp$-group scheme $\underline{\Zp^{\times}}$.

    Similarly, the morphism of stacks
    \[
        \catpolchdisps_{g, J}^{(m, n)} \lra \cathpolchdisps_{g, J}^{(m, n)}
    \]
    is a torsor under the smooth affine $\Zp$-group scheme $(L^{(m)} \Gm)^{\sigma = \id}$ that is given by $R \mapsto W_m(R)^{\times, \sigma = \id}$.
\end{lemma}

\begin{proof}
    This follows from \Cref{lem:pol-hpol-pullback} and the fact that every tuple $(I, \iota)$ consisting of an invertible $W(R)$-module $I$ together with an isomorphism $\iota \colon I^{\sigma} \to I$ can be trivialized after base changing along a faithfully flat ring homomorphism $R \to R'$.
\end{proof}

\subsection{Quotient stack descriptions} \label{subsec:quotient-stack}

We are now interested in finding quotient stack descriptions of the stack of chains of displays $\catchdisps_{h, J, d}$ and its truncated and (homogeneously) polarized variants.

Let us first return to the situation of \Cref{subsec:chains-pairs-displays}, where $h$, $J$ and $d$ are arbitrary.
Fix a $\Qp$-vector space $V$ of dimension $h$ and a tuple $(\Lambda_i)_{i \in J}$ of $\Zp$-lattices $\Lambda_i \subseteq V$ such that the following conditions are satisfied: 
\begin{itemize}
    \item
    For $i \leq j$ we have $\Lambda_i \subseteq \Lambda_j$, and the $\Zp$-module $\Lambda_j/\Lambda_i$ is of length $j - i$.

    \item
    $\Lambda_{i + h} = p^{-1} \Lambda_i$.
\end{itemize}
For $i \leq j$ we write $\rho_{i, j} \colon \Lambda_i \to \Lambda_j$ for the inclusion, and we also write $\theta_i \colon \Lambda_i \to \Lambda_{i + h}$ for the isomorphism given by multiplication with $p^{-1}$.

We then have the associated parahoric $\Zp$-group scheme
\[
    \GL \roundbr*{(\Lambda_i)_i} \coloneqq \Aut \roundbr*{(\Lambda_i)_i, \left(\rho_{i, j}\right)_{i, j}, (\theta_i)_i}
\]
with generic fiber $\GL((\Lambda_i)_i)_{\Qp} = \GL(V)$ and the \emph{local model}
\[
    \bbM^{\loc, \GL} \colon R \longmapsto \left\{(C_i)_i\;\Bigg\vert\;\begin{gathered} \text{$C_i \subseteq \Lambda_{i, R}$ a direct summand of rank $d$} \\ \text{such that $\rho_{i, j}(C_i) \subseteq C_j$ and $\theta_i(C_i) = C_{i + h}$} \end{gathered}\right\}
\]
that is a projective $\Zp$-scheme with a natural $\GL((\Lambda_i)_i)$-action whose generic fiber is the Grassmannian of $d$-dimensional subspaces of $V$; see \cite[Theorem 3.11 and Definition 3.27]{rapoport-zink-96}.

In fact, by \cite[Theorem 4.25]{goertz-flatness-local-models-linear}, $\bbM^{\loc, \GL}$ is flat over $\Zp$, its special fiber is reduced, and the irreducible components of its special fiber are normal with rational singularities.
We write $\rmM^{\loc, \GL}$ for the $p$-completion of $\bbM^{\loc, \GL}$.

Also note  that the reductive quotient $\GL((\Lambda_i)_i)_{\Fp}^{\rdt}$ of the special fiber of $\GL((\Lambda_i)_i)$ identifies with
\[
    \GL\roundbr*{(\Lambda_e)_e} \coloneqq \Aut \roundbr*{(\Lambda_e)_e, (\theta_e)_e} \cong \prod_{e \in \calE/h\ZZ} \GL(\Lambda_e),
\]
where
\[
    \Lambda_e \coloneqq \ker \roundbr*{\rho_{i, j} \colon \Lambda_i/p \Lambda_i \lra \Lambda_j/p \Lambda_j}
\]
and $\theta_e \colon \Lambda_e \to \Lambda_{e + h}$ is the isomorphism induced by $\theta_i$ (for $e = (i, j) \in \calE$).

\begin{lemma} \label{lem:chains-torsors}
    We have natural equivalences of groupoids
    \[
        \curlybr*{\text{chains over $R$}} \lra \curlybr*{\text{$L^+ \GL \roundbr*{(\Lambda_i)_i}$-torsors over $R$}}, 
    \]
%    and
    \[
        \curlybr*{\text{$n$-truncated chains over $R$}} \lra \curlybr*{\text{$L^{(n)} \GL\roundbr*{(\Lambda_i)_i}$-torsors over $R$}}
    \]
    for $n \neq 1 \blank \rdt$.
    When $R$ is of characteristic $p$, we also have a natural equivalence of groupoids
    \[
        \curlybr*{\text{$(1 \blank \rdt)$-truncated chains over $R$}} \lra \curlybr*{\text{$\GL\roundbr*{(\Lambda_e)_e}$-torsors over $R$}}.
    \]
\end{lemma}

\begin{proof}
    This follows from \cite[Theorem 3.11]{rapoport-zink-96} and \cite[Lemma 2.12]{bueltel-hedayatzadeh}.
\end{proof}

\begin{lemma} \label{lem:pairs-quotient-stack}
    We have natural equivalences of groupoids
    \[
        \curlybr*{\text{chains of pairs over $R$}} \lra \squarebr*{L^+ \GL \roundbr*{(\Lambda_i)_i} \backslash \rmM^{\loc, \GL}}(R), 
    \]
%    and
    \[
        \curlybr*{\text{$m$-truncated chains of pairs over $R$}} \lra \squarebr*{L^{(m)} \GL \roundbr*{(\Lambda_i)_i} \backslash \rmM^{\loc, \GL}}(R).
    \]
\end{lemma}

\begin{proof}
    An object in $\squarebr{L^+ \GL((\Lambda_i)_i) \backslash \rmM^{\loc, \GL}}(R)$ is given by a $\GL((\Lambda_i)_i)$-torsor $\calP$ over $W(R)$ together with a $\GL((\Lambda_i)_i)$-equivariant morphism $q \colon \calP_R \to \rmM^{\loc, \GL}$.
    By \Cref{lem:chains-torsors}, $\calP$ corresponds to a chain $((M_i)_i, (\rho_{i, j})_{i, j}, (\theta_i)_i)$ over $R$, and from the definition of $\rmM^{\loc, \GL}$, we see that $q$ corresponds to a tuple $(C_i)_i$ of direct summands $C_i \subseteq M_i/I_R M_i$ of rank $d$ that are compatible under $\rho_{i, j}$ and $\theta_i$.
    Giving such a tuple $(C_i)_i$ is evidently equivalent to giving a tuple $(M_{i, 1})_i$ of $W(R)$-submodules $M_{i, 1} \subseteq M_i$ such that $((M_i)_i, (\rho_{i, j})_{i, j}, (\theta_i)_i, (M_{i, 1})_i)$ is a chain of pairs over $R$.
\end{proof}

\begin{definition} \label{def:m-loc-gl-+}
    Let
    \[
        \rmM^{\loc, \GL, +} \lra \rmM^{\loc, \GL}
    \]
    be the $L^+ \GL((\Lambda_i)_i)$-equivariant $L^+ \GL((\Lambda_i)_i)$-torsor that corresponds to the functor from \Cref{prop:m-1-tilde-chains} under the equivalences from \Cref{lem:chains-torsors,lem:pairs-quotient-stack}.

    Explicitly, a point in $\rmM^{\loc, \GL, +}(R)$ is given by a chain of pairs $((M_i)_i, (\rho_{i, j})_{i, j}, (\theta_i)_i, (M_{i, 1})_{i, 1})$ over $R$ together with trivializing isomorphisms
    \[
        \roundbr*{\left(\widetilde{M_{i, 1}}\right)_i, \left(\widetilde{\rho_{i, j}}\right)_{i, j}, \left(\widetilde{\theta_i}\right)_i} \lra \roundbr*{\left(\Lambda_{i, W(R)}\right)_i, \left(\rho_{i, j}\right)_{i, j}, (\theta_i)_i}
    \]
    and
    \[
        \roundbr*{(M_i)_i, \left(\rho_{i, j}\right)_{i, j}, (\theta_i)_i} \lra \roundbr*{\left(\Lambda_{i, W(R)}\right)_i, \left(\rho_{i, j}\right)_{i, j}, (\theta_i)_i}.
    \]
    The group $L^+ \GL((\Lambda_i)_i) \times L^+ \GL((\Lambda_i)_i)$ acts on $\rmM^{\loc, \GL, +}$ by changing the two trivializations; the projection $\rmM^{\loc, \GL, +} \to \rmM^{\loc}$ is then a torsor with respect to the action of the first copy of $L^+ \GL((\Lambda_i)_i)$ and equivariant with respect to the action of the second copy of $L^+ \GL((\Lambda_i)_i)$.

    For $n \neq 1 \blank \rdt$ we write
    \[
        \rmM^{\loc, \GL, (n)} \lra \rmM^{\loc}
    \]
    for the reduction of $\rmM^{\loc, \GL, +}$ to an $L^{(n)} \GL((\Lambda_i)_i)$-torsor.
    By \Cref{prop:m-1-tilde-chains} the action of the second copy of $L^+ \GL((\Lambda_i)_i)$ on $\rmM^{\loc, \GL, (n)}$ factors through $L^{(m)} \GL((\Lambda_i)_i)$.

    Finally we also write
    \[
        \rmM^{\loc, \GL, (1 \blank \rdt)} \lra \rmM^{\loc}_{\Fp}
    \]
    for the reduction of $\rmM^{\loc, \GL, +}_{\Fp}$ to a $\GL((\Lambda_e)_e)$-torsor.
\end{definition}

\begin{proposition} \label{lem:disps-quotient-stack}
    We have equivalences of stacks over $\Spf(\Zp)$
    \[
        \catchdisps_{h, J, d} \lra \squarebr*{\roundbr*{L^+ \GL ((\Lambda_i)_i)}_{\Delta} \backslash \rmM^{\loc, \GL, +}}
        \quad \text{and} \quad
        \catchdisps_{h, J, d}^{(m, n)} \lra \squarebr*{\roundbr*{L^{(m)} \GL ((\Lambda_i)_i)}_{\Delta} \backslash \rmM^{\loc, \GL, (n)}},
    \]
    where the subscript $\Delta$ indicates that we take the quotient by the diagonal action.

    In particular, $\catchdisps_{h, J, d}^{(m, n)}$ is a $p$-adic formal algebraic stack of finite presentation over $\Spf(\Zp)$ $($an algebraic stack of finite presentation over $\Spec(\Fp)$ when $n = 1 \blank \rdt)$, and for $(m, n) \leq (m', n')$ the morphism $\catchdisps_{h, J, d}^{(m', n')} \to \catchdisps_{h, J, d}^{(m, n)}$ is smooth.
\end{proposition}

\begin{proof}
    By reformulating the definition of $\rmM^{\loc, \GL, +}$, we arrive at a natural identification
    \[
        \rmM^{\loc, \GL, +}(R) \cong \left\{\roundbr*{(M_{i, 1})_i, (\Psi_i)_i}\;\Bigg\vert\;\begin{gathered} \text{$\roundbr*{(\Lambda_{i, W(R)})_i, \left(\rho_{i, j}\right)_{i, j}, (\theta_i)_i, (M_{i, 1})_i, (\Psi_i)_i}$} \\ \text{is a chain of displays over $R$} \end{gathered}\right\}, 
    \]
    and under this identification the action of $L^+ \GL((\Lambda_i)_i) \times L^+ \GL((\Lambda_i)_i)$ is given by
    \[
        (k_1, k_2) . \roundbr*{(M_{i, 1})_i, (\Psi_i)_i} = \roundbr*{k_2 \cdot (M_{i, 1})_i, k_1 \cdot (\Psi_i)_i \cdot \widetilde{k_2}^{-1}},
    \]
    where we view $k_2$ as an isomorphism of chains of pairs
    \[
        k_2 \colon \roundbr*{\left(\Lambda_{i, W(R)}\right)_i, \left(\rho_{i, j}\right)_{i, j}, (\theta_i)_i, (M_{i, 1})_i} \lra \roundbr*{\left(\Lambda_{i, W(R)}\right)_i, \left(\rho_{i, j}\right)_{i, j}, (\theta_i)_i, k_2 \cdot (M_{i, 1})_i}
    \]
    in order to make sense of $\widetilde{k_2}$.
    From this the claim now readily follows.
\end{proof}

\begin{remark}[Relation with local shtukas] \label{rmk:relation-shtukas}
    Let us recall the notions of \emph{$($restricted$)$ local shtukas in mixed characteristic} from \cite{xiao-zhu};  see also \cite[Section 4]{shen-yu-zhang}.
    We warn the reader that our normalizations are slightly different from the ones in the references.

    Fix a connected reductive group $G$ over $\Qp$ and a parahoric $\Zp$-group scheme $\calG$ with generic fiber $G$.
    We write $\bbL^+ \calG$ for the perfection of the special fiber of $L^+ \calG$ and use similar notation for the truncated positive loop groups.
    We also define the \emph{$($Witt vector$)$ loop group} $\bbL G$ as the functor
    \[
        \bbL G \colon \curlybr*{\text{perfect rings of characteristic $p$}} \lra \curlybr*{\text{groups}}, \quad R \longmapsto G \roundbr*{W(R)[1/p]}.
    \]
    We then have the \emph{affine flag variety}
    \[
        \Fl_{\calG} \coloneqq \bbL^+ \calG \backslash \bbL G, 
    \]
    that is an ind-projective ind-perfect scheme over $\Fp$ by \cite[Corollary 9.6]{bhatt-scholze}.
    We equip $\bbL G$ with the action
    \[
        \roundbr*{\bbL^+ \calG \times \bbL^+ \calG} \times \bbL G \lra \bbL G, \quad \roundbr*{(k_1, k_2), g} \longmapsto k_1 g \sigma(k_2)^{-1}
    \]
    and $\Fl_{\calG}$ with the corresponding action
    \[
        \bbL^+ \calG \times \Fl_{\calG} \lra \Fl_{\calG}, \quad \roundbr*{k, \bbL^+ \calG \cdot g} \longmapsto \bbL^+ \calG \cdot g \sigma(k)^{-1}.
    \]
    This makes the projection $\bbL G \to \Fl_{\calG}$ into an $\bbL^+ \calG$-equivariant $\bbL^+ \calG$-torsor.

    Let us also fix a minuscule conjugacy class $\mu$ of cocharacters of $G_{\Qpbar}$ (where $\Qpbar$ is a fixed algebraic closure of $\Qp$) and assume for simplicity that it is defined over $\Qp$.
    Associated to $\mu$ we then have the \emph{admissible locus} $\calA_{\calG, \mu} \subseteq \Fl_{\calG}$ that is the descent to $\Fp$ of the closed perfect subscheme
    \[
        \bigcup_{w \in \Adm(\mu)_{\calG}} \Fl_{\calG, w} \subseteq \Fl_{\calG, \Fpbar},
    \]
    where $\Adm(\mu)_{\calG}$ denotes the $\mu$-admissible set for $\calG$, see \cite[Section 2]{he-rapoport}, and $\Fl_{\calG, w} \subseteq \Fl_{\calG, \Fpbar}$ is the $\bbL^+ \calG$-orbit corresponding to $w$.
    Write
    \[
        \calA_{\calG, \mu}^+ \subseteq \bbL G
    \]
    for the preimage of $\calA_{\calG, \mu}$ under the projection $\bbL G \to \Fl_{\calG}$, and write
    \[
        \calA_{\calG, \mu}^{(n)} \lra \calA_{\calG, \mu}
    \]
    for the reduction of $\calA_{\calG, \mu}^+$ to an $\bbL^{(n)} \calG$-torsor, where we formally set $\bbL^{(1 \blank \rdt)} \calG$ to be the perfection of the reductive quotient $\calG_{\Fp}^{\rdt}$.
    By \cite[proof of Theorem 3.16]{p-adic-local-models} the action of $\bbL^+ \calG$ on $\calA_{\calG, \mu}$ factors through $\calG_{\Fp}^{\pf}$, and the argument in \cite[proof of Lemma 4.4.2]{shen-yu-zhang} then shows that the $\bbL^+ \calG$-equivariant structure on $\calA_{\calG, \mu}^{(n)} \to \calA_{\calG, \mu}$ factors through $\bbL^{(m)} \calG$.

    The stack of \emph{local shtukas for $(\calG, \mu)$} is then defined to be
    \[
        \Sht_{\calG, \mu}^{\loc} \coloneqq \squarebr*{\roundbr*{\bbL^+ \calG}_{\Delta} \backslash \calA_{\calG, \mu}^+}, 
    \]
    and the stack of \emph{$(m, n)$-restricted local shtukas for $(\calG, \mu)$} is defined similarly as
    \[
        \Sht_{\calG, \mu}^{\loc, (m, n)} \coloneqq \squarebr*{\roundbr*{\bbL^{(m)} \calG}_{\Delta} \backslash \calA_{\calG, \mu}^{(n)}}.
    \]

    Now let us specialize to the situation where $\calG = \GL((\Lambda_i)_i)$ and $\mu = \mu_d$ is the conjugacy class of those cocharacters of $\GL(V)_{\Qpbar}$ that induce a weight decomposition $V_{\Qpbar} = V_1 \oplus V_0$ with $V_1$ of dimension $d$.
    Then we have the explicit description
    \[
        \calA_{\GL((\Lambda_i)_i), \mu_d}^+(R) = \left\{g \in \GL(V)\roundbr*{W(R)[1/p]}\;\Bigg\vert\;\begin{gathered} \text{$p \Lambda_{i, W(R)} \subseteq g \Lambda_{i, W(R)} \subseteq \Lambda_{i, W(R)}$ and $\Lambda_{i, W(R)}/g \Lambda_{i, W(R)}$} \\ \text{is a finite projective $R$-module of rank $d$ for all $i \in J$} \end{gathered}\right\};
    \]
    this follows from the equality between the $\mu$-admissible and the $\mu$-permissible set \cite[Theorem 3.5]{kottwitz-rapoport}; see also \cite[Section 4.3]{goertz-flatness-local-models-linear}.
    Using the description of $\rmM^{\loc, \GL, +}$ from the proof of \Cref{lem:disps-quotient-stack} and applying \Cref{rmk:relation-dieudonne}, we obtain an $(\bbL^+ \GL((\Lambda_i)_i) \times \bbL^+ \GL((\Lambda_i)_i))$-equivariant isomorphism
    \[
        \roundbr*{\rmM^{\loc, \GL, +}}_{\Fp}^{\pf} \lra \calA_{\GL((\Lambda_i)_i), \mu_d}^+, \quad \roundbr*{(M_{i, 1})_i, (\Psi_i)_i} \longmapsto F_M,
    \]
    where $F_M \colon W(R)[1/p] \otimes_{\Qp} V \cong (W(R)[1/p] \otimes_{\Qp} V)^{\sigma} \to W(R)[1/p] \otimes_{\Qp} V$ is the Frobenius of the chain of displays $M = ((\Lambda_{i, W(R)})_i, (\rho_{i, j})_{i, j}, (\theta_i)_i, (M_{i, 1})_i, (\Psi_i)_i)$.
    Consequently, we obtain equivalences
    \[
        \roundbr*{\catchdisps_{h, J, d}}_{\Fp}^{\pf} \lra \Sht_{\GL((\Lambda_i)_i), \mu_d}^{\loc}
        \quad \text{and} \quad
        \roundbr*{\catchdisps_{h, J, d}^{(m, n)}}_{\Fp}^{\pf} \lra \Sht_{\GL((\Lambda_i)_i), \mu_d}^{\loc, (m, n)}
    \]
    between the perfected special fibers of the stacks of chains of displays and the stacks of local shtukas for $(\GL((\Lambda_i)_i), \mu_d)$.
\end{remark}

\begin{remark} \label{rmk:chains-lau-zink}
    One could also define a stack
    \[
        \catchdisps_{h, J, d}^{\LZ, (n)} 
    \]
    of \emph{$n$-truncated chains of Lau--Zink-displays over $R$} (see \Cref{rmk:relation-lau-zink}) in the straightforward way.
    This is done in \cite{hesse} for $n = 1$.

    However, this definition is pathological away from the hyperspecial case $J = r + h\ZZ$, even when restricting to perfect rings of characteristic $p$.
    The truncation morphism
    \[
        \catchdisps_{h, J, d} \lra \catchdisps_{h, J, d}^{\LZ, (n)}
    \]
    is not surjective, and its topological image is not even locally closed (in particular, \cite[Lemma 2.51]{hesse} is false) as the following example shows.

    Let $(h, J, d) = (2, \ZZ, 1)$ and define an $n$-truncated chain $M = ((M_i, M_{i, 1}, \Psi_i)_i, (\rho_{i, j})_{i, j}, (\theta_i)_i)$ of Lau--Zink-displays over $R = \Fp[x^{1/p^{\infty}}, y^{1/p^{\infty}}]$ as follows: 
    \begin{itemize}
        \item
        $(M_i, M_{i, 1}) \coloneqq (M, M_1)$ is the standard $n$-truncated Lau--Zink-pair of type $(2, 1)$ with fixed normal decomposition $(L, T) = (W_n(R), W_n(R))$.

        \item
        $\theta_i \coloneqq \id_{(M, M_1)}$.

        \item
        \[
            \rho_{0, 1} \coloneqq
            \begin{pmatrix}
                [x] & 1 \\ p + [y] \cdot p^n & 0
            \end{pmatrix}, \;
            \rho_{1, 2} \coloneqq
            \begin{pmatrix}
                0 & 1 \\ p + [y] \cdot p^n & -[x]
            \end{pmatrix}
            \in \End_R\roundbr*{(M, M_1)} \cong
            \begin{pmatrix}
                W_n(R) & W_n(R) \\ I_{n + 1, R} & W_n(R)
            \end{pmatrix}.
        \]
        This uniquely determines $\rho_{i, j}$ for all $i, j$.

        \item
        One can now check  that both $((\widetilde{M_{i, 1}})_i, (\widetilde{\rho_{i, j}})_{i, j}, (\widetilde{\theta_i})_i)$ and $((M_i)_i, (\rho_{i, j}), (\theta_i)_i)$ are trivial $n$-truncated chains over $R$.
        We let $(\Psi_i)_i$ be an arbitrary isomorphism between them.
    \end{itemize}
    We now claim that the set-theoretic locus $Z \subseteq \abs*{\Spec(R)}$ where $M$ lifts to a chain of displays is given by the constructible subset
    \[
        Z = D(x) \cup V(y)
    \]
    that is not locally closed.

    From the smoothness of $\GL((\Lambda_i)_i)$, it follows that an $n$-truncated chain of Lau--Zink-displays lifts to a chain of displays if and only if its underlying $n$-truncated chain of Lau--Zink-pairs lifts.
    \begin{itemize}
        \item
        Consider the $n$-truncated chain of Lau--Zink pairs $M' \coloneqq ((M_i, M_{i, 1})_i, (\rho'_{i, j})_{i, j}, (\theta_i)_i)$ over $\Fp[x^{1/p^{\infty}}]$, where the $\rho_{i, j}'$ are defined by setting
        \[
            \rho'_{0, 1} \coloneqq
            \begin{pmatrix}
                [x] & 1 \\ p & 0
            \end{pmatrix}
            \quad \text{and} \quad
            \rho'_{1, 2} \coloneqq
            \begin{pmatrix}
                0 & 1 \\ p & -[x]
            \end{pmatrix}.
        \]
        Then $M'$ lifts to a chain of pairs that is defined by the same expression.

        Now $M$ (or rather its underlying $n$-truncated chain of Lau--Zink-pairs) clearly becomes isomorphic to $M'$ after base changing along $R \to (R/yR)^{\pf} \cong \Fp[x^{1/p^{\infty}}]$, so  we have $V(y) \subseteq Z$.
        Moreover, we also have an isomorphism $(\alpha_i)_i \colon M'_{R[x^{-1}]} \to M_{R[x^{-1}]}$ that is defined by
        \[
            \alpha_0 \coloneqq
            \begin{pmatrix}
                1 & 0 \\ 0 & 1
            \end{pmatrix}
            \quad \text{and} \quad
            \alpha_1 \coloneqq
            \begin{pmatrix}
                1 & 0 \\ [x^{-1} y] \cdot p^n & 1
            \end{pmatrix},
        \]
        so  we also have $D(x) \subseteq Z$.

        \item
        Now let $R \to k$ be a ring homomorphism from $R$ into a perfect field $k$ such that the image of $x$ in $k$ vanishes while the image of $y$ in $k$ is non-zero, and suppose that $M_k$ lifts to a chain of pairs $M^{\lift}$.
        After compatibly lifting normal decompositions, we obtain that
        $M^{\lift}$ is given by
        \[
            \rho^{\lift}_{0, 1} \coloneqq
            \begin{pmatrix}
                p^n \cdot a & 1 + p^n \cdot b \\ p + [y] \cdot p^n + p^{n + 1} \cdot c & p^n \cdot d
            \end{pmatrix}
            \quad \text{and} \quad
            \rho^{\lift}_{1, 2} \coloneqq
            \begin{pmatrix}
                p^n \cdot e & 1 + p^n \cdot f \\ p + [y] \cdot p^n + p^{n + 1} \cdot g & p^n \cdot h
            \end{pmatrix}
        \]
        for some $a, \dotsc, h \in W(k)$.
        But then the upper-left entry of the matrix representing $\rho^{\lift}_{1, 2} \circ \rho^{\lift}_{0, 1}$ is given by
        \[
        (p^n \cdot e) \cdot (p^n \cdot a) + (1 + p^n \cdot f) \cdot \roundbr*{p + [y] \cdot p^n + p^{n + 1} \cdot c} \equiv p + [y] \cdot p^n \quad \mod p^{n + 1}.
        \]
        This gives a contradiction as $\rho^{\lift}_{1, 2} \circ \rho^{\lift}_{0, 1}$ really should be multiplication by $p$.
    \end{itemize}
    This proves the claim.
\end{remark}

Let us now come back to the situation of \Cref{subsec:polarized-chains},
% query comma
where $h = 2g$ is even, $d = g$ and $-J = J$.
Fix a polarization $\lambda \colon V \to V^{\vee}$ that restricts to isomorphisms $\lambda_i \colon \Lambda_i \to \Lambda_{- i}^{\vee}$ for all $i \in J$.
We also set $\Gamma \coloneqq \Zp$ and sometimes think of $\lambda$ as a homogeneous polarization $\lambda \colon V \to \Gamma[1/p] \otimes_{\Qp} V^{\vee}$ that restricts to isomorphisms $\lambda_i \colon \Lambda_i \to \Gamma \otimes_{\Zp} \Lambda_{- i}^{\vee}$.

We then again have associated parahoric $\Zp$-group schemes
\[
    \Sp \roundbr*{(\Lambda_i)_i} \coloneqq \Aut \roundbr*{(\Lambda_i)_i, \left(\rho_{i, j}\right)_{i, j}, (\theta_i)_i, (\lambda_i)_i}
\]
and
\[
    \GSp \roundbr*{(\Lambda_i)_i} \coloneqq \Aut \roundbr*{(\Lambda_i)_i, \left(\rho_{i, j}\right)_{i, j}, (\theta_i)_i, \Gamma, (\lambda_i)_i}
\]
with generic fibers $\Sp((\Lambda_i)_i)_{\Qp} = \Sp(V)$ and $\GSp((\Lambda_i)_i)_{\Qp} = \GSp(V)$ that fit into a short exact sequence
\[
    1 \lra \Sp \roundbr*{(\Lambda_i)_i} \lra \GSp \roundbr*{(\Lambda_i)_i} \lra \Gm \lra 1,
\]
where the second arrow is given by sending a symplectic similitude to the similitude factor.
We also have the \emph{local model}
\[
    \bbM^{\loc, \Sp} = \bbM^{\loc, \GSp} \colon R \longmapsto \left\{(C_i)_i \in \bbM^{\loc, \GL}(R)\;\Big\vert\;\lambda_i(C_i) = C_{-i}^{\perp}\right\}
\]
that is a closed subscheme of $\bbM^{\loc, \GL}$ stable under the action of $\GSp((\Lambda_i)_i) \subseteq \GL((\Lambda_i)_i)$ whose generic fiber is the Grassmannian of Lagrangian subspaces of $V$ (see \cite[Theorem 3.16 and Definition 3.27]{rapoport-zink-96}).

Again, by \cite[Theorem 2.1]{goertz-flatness-local-models-symplectic}, $\bbM^{\loc, \Sp} = \bbM^{\loc, \GSp}$ is flat over $\Zp$, its special fiber is reduced, and the irreducible components of its special fiber are normal with rational singularities.
We write $\rmM^{\loc, \Sp} = \rmM^{\loc, \GSp}$ for the $p$-completion of $\bbM^{\loc, \Sp} = \bbM^{\loc, \GSp}$.

The reductive quotients $\Sp((\Lambda_i)_i)_{\Fp}^{\rdt}$ and $\GSp((\Lambda_i)_i)_{\Fp}^{\rdt}$ of the special fibers of $\Sp((\Lambda_i)_i)$ and $\GSp((\Lambda_i)_i)$ identify with
\[
    \Sp \roundbr*{(\Lambda_e)_e} \coloneqq \Aut \roundbr*{(\Lambda_e)_e, (\theta_e)_e, (\lambda_e)_e}
    \quad \text{and} \quad
    \GSp \roundbr*{(\Lambda_e)_e} \coloneqq \Aut \roundbr*{(\Lambda_e)_e, (\theta_e)_e, \Gamma/p\Gamma, (\lambda_e)_e},
\]
where $\lambda_e \colon \Lambda_e \to \Lambda_{-e}^{\vee} \cong (\Gamma/p \Gamma) \otimes_{\Fp} \Lambda_{-e}^{\vee}$ is the isomorphism induced by $(\lambda_i)_i$.
These groups have the following explicit description: 
\begin{itemize}
    \item
    Suppose that $0, g \in J$.
    Then we have
    \[
        \Sp\roundbr*{(\Lambda_e)_e} = \prod_{e \in \calE/(h\ZZ, \pm)} \GL(\Lambda_e),
        \quad
        \GSp \roundbr*{(\Lambda_e)_e} = \Gm \times \prod_{e \in \calE/(h\ZZ, \pm)} \GL(\Lambda_e).
    \]

    \item
    Suppose that $0 \notin J$ and $g \in J$, and write $e_0 \in \calE$ for the edge $(i, j)$ with $i < 0 < j$.
    Then we have
    \[
        \Sp \roundbr*{(\Lambda_e)_e} = \Sp\left(\Lambda_{e_0}\right) \times \prod_{e \in (\calE/(h \ZZ, \pm)) \setminus \curlybr{e_0}} \GL(\Lambda_e),
        \quad
        \GSp \roundbr*{(\Lambda_e)_e} = \GSp\left(\Lambda_{e_0}\right) \times \prod_{e \in (\calE/(h \ZZ, \pm)) \setminus \curlybr{e_0}} \GL(\Lambda_e).
    \]

    \item
    Suppose that $0 \in J$ and $g \notin J$, and write $e_g \in \calE$ for the edge $(i, j)$ with $i < g < j$.
    Then we similarly have
    \[
        \Sp \roundbr*{(\Lambda_e)_e} = \Sp(\Lambda_{e_g}) \times \prod_{e \in (\calE/(h \ZZ, \pm)) \setminus \curlybr{e_g}} \GL(\Lambda_e),
        \quad
        \GSp \roundbr*{(\Lambda_e)_e} = \GSp(\Lambda_{e_g}) \times \prod_{e \in (\calE/(h \ZZ, \pm)) \setminus \curlybr{e_g}} \GL(\Lambda_e).
    \]

    \item
    Suppose that $0, g \notin J$, and write $e_0, e_g \in \calE$ as before.
    Then we have
    \[
        \Sp \roundbr*{(\Lambda_e)_e} = \Sp\left(\Lambda_{e_0}\right) \times \Sp\left(\Lambda_{e_g}\right) \times \prod_{e \in (\calE/(h \ZZ, \pm)) \setminus \curlybr{e_0, e_g}} \GL(\Lambda_e)
    \]
    and
    \[
        \GSp \roundbr*{(\Lambda_e)_e} = \GSp\left(\Lambda_{e_0}, \Lambda_{e_g}\right) \times \prod_{e \in (\calE/(h \ZZ, \pm)) \setminus \curlybr{e_0, e_g}} \GL(\Lambda_e),
    \]
    where $\GSp(\Lambda_{e_0}, \Lambda_{e_g}) = \GSp(\Lambda_{e_0}) \times_{\Gm} \GSp(\Lambda_{e_g})$ denotes the group of tuples of symplectic similitudes of $\Lambda_{e_0}$ and $\Lambda_{e_g}$ with the same similitude factor.
\end{itemize}

\begin{lemma} \label{lem:polarized-chains-torsors}
    We have natural equivalences of groupoids
    \[
        \curlybr*{\,\text{polarized chains over $R$}}\lra \curlybr*{\text{$L^+ \Sp \roundbr*{(\Lambda_i)_i}$-torsors over $R$}}
    \]
    and
    \[
        \curlybr*{\text{$n$-truncated polarized chains over $R$}} \lra \curlybr*{\text{$L^{(n)} \Sp \roundbr*{(\Lambda_i)_i}$-torsors over $R$}}
    \]
    for $n \neq 1 \blank \rdt$.
    When $R$ is of characteristic $p$, we also have a natural equivalence of groupoids
    \[
        \curlybr*{\text{$(1 \blank \rdt)$-truncated polarized chains over $R$}} \lra \curlybr*{\text{$\Sp \roundbr*{(\Lambda_e)_e}$-torsors over $R$}}.
    \]

    Similarly, we also have a natural equivalence between homogeneously polarized chains over $R$ and $L^+ \GSp((\Lambda_i)_i)$-torsors over $R$ and truncated variants.
\end{lemma}

\begin{proof}
    The claims for (truncated) polarized chains follow from \cite[Theorem 3.16]{rapoport-zink-96} and \cite[Lemma~2.12]{bueltel-hedayatzadeh}.

    To prove the claim for homogeneously polarized chains, we need to see that every such homogeneously polarized chain $((M_i)_i, (\rho_{i, j})_{i, j}, (\theta_i)_i, I, (\lambda_i)_i)$ over $R$ can be trivialized fpqc-locally on $\Spec(R)$.
    We can certainly find a local trivialization of $I$ so that the homogeneously polarized chain lifts to a polarized chain that can be trivialized locally by the first part.
\end{proof}

\begin{lemma} \label{lem:polarized-pairs-quotient-stack}
    We have natural equivalences of groupoids
    \[
        \curlybr*{\,\text{polarized chains of pairs over $R$}} \lra \squarebr*{L^+ \Sp \roundbr*{(\Lambda_i)_i} \backslash \rmM^{\loc, \Sp}}(R) 
    \]
    and
    \[
        \curlybr*{\text{$m$-truncated polarized chains of pairs over $R$}} \lra \squarebr*{L^{(m)} \Sp \roundbr*{(\Lambda_i)_i} \backslash \rmM^{\loc, \Sp}}(R).
    \]
    
    We also have analogous equivalences for $(m$-truncated\,$)$ homogeneously polarized chains of pairs.
\end{lemma}

\begin{proof}
    This follows in the same way as \Cref{lem:pairs-quotient-stack}.
\end{proof}

\begin{definition} \label{def:m-loc-sp-+}
    Let
    \[
        \rmM^{\loc, \Sp, +} \lra \rmM^{\loc, \Sp}
    \]
    be the $L^+ \Sp((\Lambda_i)_i)$-equivariant $L^+ \Sp((\Lambda_i)_i)$-torsor corresponding to the functor from \Cref{rmk:m-1-tilde-polarized-chains} for polarized chains under the equivalences from \Cref{lem:polarized-chains-torsors,lem:polarized-pairs-quotient-stack} (as in \Cref{def:m-loc-gl-+}).

    Similarly, let
    \[
        \rmM^{\loc, \GSp, +} \lra \rmM^{\loc, \GSp}
    \]
    be the $L^+ \GSp((\Lambda_i)_i)$-equivariant $L^+ \GSp((\Lambda_i)_i)$-torsor corresponding to the analogous functor from \Cref{rmk:m-1-tilde-polarized-chains} for homogeneously polarized chains.
    Note that this agrees with the base change of the $L^+ \Sp((\Lambda_i)_i)$-torsor $\rmM^{\loc, \Sp, +} \to \rmM^{\loc, \Sp} = \rmM^{\loc, \GSp}$ along the morphism of $\Zp$-group schemes $L^+ \Sp((\Lambda_i)_i) \to L^+ \GSp((\Lambda_i)_i)$.

    We also use the notation $\rmM^{\loc, \Sp, (n)}$ and $\rmM^{\loc, \GSp, (n)}$ as in \Cref{def:m-loc-gl-+}.
\end{definition}

\begin{proposition} \label{lem:polarized-disps-quotient-stack}
    We have equivalences of stacks over $\Spf(\Zp)$
    \[
        \catpolchdisps_{g, J} \lra \squarebr*{\roundbr*{L^+ \Sp ((\Lambda_i)_i)}_{\Delta} \backslash \rmM^{\loc, \Sp, +}},
        \quad
        \catpolchdisps_{g, J}^{(m, n)} \lra \squarebr*{\roundbr*{L^{(m)} \Sp ((\Lambda_i)_i)}_{\Delta} \backslash \rmM^{\loc, \Sp, (n)}}
    \]
    and
    \[
        \cathpolchdisps_{g, J} \lra \squarebr*{\roundbr*{L^+ \GSp ((\Lambda_i)_i)}_{\Delta} \backslash \rmM^{\loc, \GSp, +}},
        \;\,
        \cathpolchdisps_{g, J}^{(m, n)} \lra \squarebr*{\roundbr*{L^{(m)} \GSp ((\Lambda_i)_i)}_{\Delta} \backslash \rmM^{\loc, \GSp, (n)}}.
    \]

    In particular, $\catpolchdisps_{g, J}^{(m, n)}$  and $\cathpolchdisps_{g, J}^{(m, n)}$ are $p$-adic formal algebraic stacks of finite presentation over $\Spf(\Zp)$ $($algebraic stacks of finite presentation over $\Spec(\Fp)$ when $n = 1 \blank \rdt)$, and for $(m, n) \leq (m', n')$ the morphisms $\catpolchdisps_{g, J}^{(m', n')} \to \catpolchdisps_{g, J}^{(m, n)}$ and $\cathpolchdisps_{g, J}^{(m', n')} \to \cathpolchdisps_{g, J}^{(m, n)}$ are smooth.
\end{proposition}

\begin{proof}
    The proof is analogous to that of \Cref{lem:disps-quotient-stack}.
\end{proof}

\begin{remark}[Relation with local shtukas, continued] \label{rmk:relation-shtukas-2}
    Similarly to what is  discussed in \Cref{rmk:relation-shtukas}, we have equivalences
    \[
        \roundbr*{\cathpolchdisps_{g, J}}_{\Fp}^{\pf} \lra \Sht_{\GSp((\Lambda_i)_i), \mu_g}^{\loc}
        \quad \text{and} \quad
        \roundbr*{\cathpolchdisps_{g, J}^{(m, n)}}_{\Fp}^{\pf} \lra \Sht_{\GSp((\Lambda_i)_i), \mu_g}^{\loc, (m, n)},
    \]
    where $\mu_g$ is the conjugacy class of those cocharacters of $\GSp(V)_{\Qpbar}$ that induce a weight decomposition $V_{\Qpbar} = V_1 \oplus V_0$ with $V_0, V_1 \subseteq V_{\Qpbar}$ Lagrangian.

    The non-homogeneously polarized case, however, does not quite fit into the group-theoretic framework of local shtukas.
    The problem is essentially that given a polarized chain of displays $M$ over $R$, the corresponding Frobenius $F_M \colon M[1/p]^{\sigma} \to M[1/p]$ is not a symplectic isomorphism; it only preserves the symplectic form up to the scalar $p$.
    The situation can be remedied as follows.

    The similitude factor morphism $\bbL \GSp(V) \to \bbL \Gm$ restricts to a faithfully flat morphism
    \[
        \calA_{\GSp((\Lambda_i)_i), \mu_g}^+ \lra p \cdot \bbL^+ \Gm, 
    \]
    and we define $\calA_{\Sp((\Lambda_i)_i), \mu_g}^+$ to be the fiber of this morphism over $p \in (p \cdot \bbL^+ \Gm)(\Fp)$.
    Then the isomorphism $(\rmM^{\loc, \GSp, +})_{\Fp}^{\pf} \to \calA_{\GSp((\Lambda_i)_i), \mu_g}^+$ restricts to an $(\bbL^+ \Sp((\Lambda_i)_i) \times \bbL^+ \Sp((\Lambda_i)_i))$-equivariant isomorphism
    \[
        \roundbr*{\rmM^{\loc, \Sp, +}}_{\Fp}^{\pf} \lra \calA_{\Sp((\Lambda_i)_i), \mu_g}^+, 
    \]
    so  we obtain equivalences
    \[
        \roundbr*{\catpolchdisps_{g, J}}_{\Fp}^{\pf} \lra \Sht_{\Sp((\Lambda_i)_i), \mu_g}^{\loc} \coloneqq \squarebr*{\roundbr*{\bbL^+ \Sp((\Lambda_i)_i)}_{\Delta} \backslash \calA_{\Sp((\Lambda_i)_i), \mu_g}^+}
    \]
    and
    \[
        \roundbr*{\catpolchdisps_{g, J}^{(m, n)}}_{\Fp}^{\pf} \lra \Sht_{\Sp((\Lambda_i)_i), \mu_g}^{\loc} \coloneqq \squarebr*{\roundbr*{\bbL^+ \Sp((\Lambda_i)_i)}_{\Delta} \backslash \calA_{\Sp((\Lambda_i)_i), \mu_g}^+}.
    \]
\end{remark}

\subsection{(Polarized) chains of $\boldsymbol{p}$-divisible groups}

Let us again first return to the situation of \Cref{subsec:chains-pairs-displays},
% query/check
where $h$, $J$ and $d$ are arbitrary.

\begin{definition}
    A \emph{chain of $p$-divisible groups $($of type $(h, J, d))$ over $R$} is a diagram $((X_i)_i, (\rho_{i, j})_{i, j})$ of $p$-divisible groups of height $h$ and dimension $d$ over $R$ of shape $J^{\op}$ such that $\rho_{i, j} \colon X_j \to X_i$ is an isogeny of height $j - i$ and $\ker(\rho_{i, i + h}) = X_{i + h}[p]$.
    We write
    \[
        \catchpdiv_{h, J, d}
    \]
    for the stack over $\Spf(\Zp)$ of chains of $p$-divisible groups.

    In the situation of \Cref{subsec:polarized-chains}, we equip the category of chains of $p$-divisible groups over $R$ with a duality by setting
    \[
        \roundbr*{(X_i)_i, \left(\rho_{i, j}\right)_{i, j}}^{\vee} \coloneqq \roundbr*{\left(X_{- i}^{\vee}\right)_i, \left(\rho_{-j, -i}^{\vee}\right)_{i, j}}
    \]
    and define a \emph{polarized chain of $p$-divisible groups $($of type $(g, J))$ over $R$} to be a polarized object $((X_i)_i, (\rho_{i, j})_{i, j}, (\lambda_i)_i)$ in the category of chains of $p$-divisible groups over $R$.
    As before we write
    \[
        \catpolchpdiv_{g, J}
    \]
    for the stack over $\Spf(\Zp)$ of polarized chains of $p$-divisible groups.
\end{definition}

\begin{proposition} \label{prop:disp-p-div-chains}
    The functor $\DD$ from \Cref{thm:disp-p-div} induces a morphism
    \[
        \catchpdiv_{h, J, d} \lra \catchdisps_{h, J, d}
    \]
    that restricts to an equivalence $\catchpdiv_{h, J, d}^{\formal} \to \catchdisps_{h, J, d}^{F \blank \nilp}$ between the substack of those chains of $p$-divisible groups $((X_i)_i, (\rho_{i, j})_{i, j})$ such that the $X_i$ are formal $p$-divisible groups and the substack of those chains of displays $((M_i)_i, (\rho_{i, j})_{i, j}, (\theta_i)_i, (M_{i, 1})_i, (\Psi_i)_i)$ such that the $(M_i, M_{i, 1}, \Psi_i)$ are $F$-nilpotent.

    In the situation of \Cref{subsec:polarized-chains}, we similarly obtain a morphism
    \[
        \catpolchpdiv_{g, J} \lra \catpolchdisps_{g, J}
    \]
    that restricts to an equivalence $\catpolchpdiv_{g, J}^{\formal} \to \catpolchdisps_{g, J}^{F \blank \nilp}$.
\end{proposition}

\begin{proof}
    Given a chain of $p$-divisible groups $((X_i)_i, (\rho_{i, j})_{i, j})$ over $R$, the condition $\ker(\rho_{i, i + h}) = X_{i + h}[p]$ implies that there exists a uniquely determined isomorphism $\theta_i \colon X_{i + h} \to X_i$ such that $p \cdot \theta_i = \rho_{i, i + h}$.
    We can then apply the functor $\DD$ to the data $X_i$, $\rho_{i, j}$, $\theta_i$, and by \Cref{rmk:chains-disps} this yields a chain of displays; here we use the last part of \Cref{thm:disp-p-div} to verify the rank condition for the homomorphism $\DD(\rho_{i, j})$.
    Thus we obtain the desired morphism $\catchpdiv_{h, J, d} \to \catchdisps_{h, J, d}$.

    The claim that this morphism restricts to an equivalence $\catchpdiv_{h, J, d}^{\formal} \to \catchdisps_{h, J, d}^{F \blank \nilp}$ follows immediately from the corresponding statement in \Cref{thm:disp-p-div}.
\end{proof}     \section{Application to the Siegel modular variety}

As in \Cref{sec:chains-pairs-disps}, $(m, n)$ denotes a tuple of positive integers with $m \geq n + 1$, where we allow $n$ to take the additional value $1 \blank \rdt$.

\subsection{The Siegel modular variety at parahoric level}

Fix a positive integer $g$ and a non-empty subset $J \subseteq \ZZ$ such that $J + 2g \ZZ = J$ and $-J = J$ as in \Cref{subsec:polarized-chains}.
Also fix an auxiliary integer $N \geq 3$ such that $p \nmid N$, and equip $(\ZZ/N\ZZ)^{2g}$ with the standard symplectic form that is represented by the matrix
\[
    \begin{pmatrix}
        0 & \widetilde{I_g} \\ - \widetilde{I_g} & 0
    \end{pmatrix},
    \quad \text{where} \; \,
    \widetilde{I_g} \coloneqq
    \begin{psmallmatrix}
        & & 1 \\
        & \iddots & \\
        1 & &
    \end{psmallmatrix}.
\]

\begin{definition}
    We define the \emph{Siegel modular variety} as the moduli problem
    \[
        \calA_{g, J, N} \colon \curlybr*{\text{$\Zp$-algebras}} \lra \curlybr*{\text{sets}}
    \]
    by setting $\calA_{g, J, N}(R)$ to be the set of isomorphism classes of tuples $((A_i)_i, (\rho_{i, j})_{i, j}, (\lambda_i)_i, \eta)$ that are given as follows: 
    \begin{itemize}
        \item
        $((A_i)_i, (\rho_{i, j})_{i, j})$ is a diagram of projective Abelian varieties of dimension $g$ over $R$ of shape $J^{\op}$ such that $\rho_{i, j} \colon A_j \to A_i$ is an isogeny of degree $p^{j - i}$ and $\ker(\rho_{i, i + 2g}) = A_{i + 2g}[p]$.

        \item
        $(\lambda_i)_i \colon ((A_i)_i, (\rho_{i, j})_{i, j}) \to ((A_{- i}^{\vee})_i, (\rho_{-j, -i}^{\vee})_{i, j})$ is a symmetric isomorphism such that for some (or equivalently every) $i \geq 0$ the symmetric isogeny
        \[
            A_i \stackrel{\lambda_i}\lra A_{-i}^{\vee} \xrightarrow{\rho_{-i, i}^{\vee}} A_i
        \]
        is a polarization.

        \item
        $\eta \colon A[N] \to \underline{(\ZZ/N\ZZ)^{2g}}$ is an isomorphism of finite \'etale $R$-group schemes such that there exists an isomorphism $\mu_N \to \underline{\ZZ/N\ZZ}$ making the diagram
        \[
        \begin{tikzcd}
            A[N] \times A[N] \ar[r] \ar[d, "\eta \times \eta"]
            & \mu_N \ar[d]
            \\
            \underline{(\ZZ/N\ZZ)^{2g}} \times \underline{(\ZZ/N\ZZ)^{2g}} \ar[r]
            & \underline{\ZZ/N\ZZ}
        \end{tikzcd}
        \]
        commutative. Here $A[N]$ is the $N$-torsion of any of the Abelian varieties $A_i$ (note that $\rho_{i, j}$ restricts to an isomorphism $A_j[N] \to A_i[N]$ because $p$ does not divide $N$), and the upper horizontal arrow is the Weil pairing with respect to $(\lambda_i)_i$.
    \end{itemize}

    We also write $\calA_{g, J, N}^{\wedge}$ for the $p$-completion of $\calA_{g, J, N}$, \textit{i.e.}~its restriction to the category of $p$-nilpotent rings.
\end{definition}

\begin{proposition}
  The Siegel modular variety
% query 
  $\calA_{g, J, N}$ is representable by a quasi-projective $\Zp$-scheme.
    Moreover, there exists a natural smooth morphism
    \[
        \calA_{g, J, N} \lra \squarebr*{\GSp\roundbr*{(\Lambda_i)_i} \backslash \bbM^{\loc, \GSp}},
    \]
    where we use the notation from \Cref{subsec:quotient-stack}.
\end{proposition}

\begin{proof}
    The representability follows from \cite[Theorem 7.9]{mumford-git}; see also \cite[Definition 6.9]{rapoport-zink-96}.
    For the second claim, see \cite[Section 7(ii)]{he-rapoport}.
\end{proof}

\begin{definition} \label{def:upsilon}
    We define the morphism $\Upsilon \colon \calA_{g, J, N}^{\wedge} \to \catpolchdisps_{g, J}$ as the composition
    \[
        \Upsilon \colon \calA_{g, J, N}^{\wedge} \lra \catpolchpdiv_{g, J} \lra \catpolchdisps_{g, J},
    \]
    where the first arrow is given by
    \[
        \roundbr*{(A_i)_i, \left(\rho_{i, j}\right)_{i, j}, (\lambda_i)_i, \eta} \longmapsto \roundbr*{(A_i[p^{\infty}])_i, \left(\rho_{i, j}\right)_{i, j}, (\lambda_i)_i}
    \]
    (this is well defined because the functor $A \mapsto A[p^{\infty}]$ from projective Abelian varieties to $p$-divisible groups is compatible with dualities up to a sign $-1$, see \cite[Proposition 1.8]{oda}), and the second arrow is the one from \Cref{prop:disp-p-div-chains}.

    We then also have the induced morphism
    \[
        \upsilon^{(m, n)} \colon \calA_{g, J, N}^{\wedge} \lra \catpolchdisps_{g, J}^{(m, n)}
    \]
    for $n \neq 1 \blank \rdt$ and
    \[
        \upsilon^{(m, 1 \blank \rdt)} \colon \roundbr*{\calA_{g, J, N}}_{\Fp} \lra \catpolchdisps_{g, J}^{(m, 1 \blank \rdt)}.
    \]
\end{definition}

\begin{remark}
   The perfection of the composition
    \[
        \roundbr*{\calA_{g, J, N}}_{\Fp} \xrightarrow{\upsilon^{(m, 1 \blank \rdt)}} \catpolchdisps_{g, J}^{(m, 1 \blank \rdt)} \lra \cathpolchdisps_{g, J}^{(m, 1 \blank \rdt)}
    \]
    yields the morphism $\upsilon_K$ from \cite[Section 4.4]{shen-yu-zhang}, at least up to the difference in normalization (see \Cref{rmk:relation-shtukas,rmk:relation-shtukas-2}).

    In particular, the fibers of this composition are precisely the EKOR strata defined in \cite[Definition~6.4]{he-rapoport}.
    More precisely, we have a natural bijection
    \[
        \abs*{\roundbr*{\cathpolchdisps^{(m, 1 \blank \rdt)}_{g, J}}_{\Fpbar}} \cong \breve{K}_{\sigma} \backslash (\breve{K}_1 \backslash X),
        \quad \text{where} \; \,
        X = \bigcup_{w \in \Adm_{g, J}} \breve{K} w \breve{K} \subseteq \GSp(V)(\Qpbrev)
    \]
    as in the introduction.
    For $x \in \breve{K}_{\sigma} \backslash (\breve{K}_1 \backslash X)$ we thus have an associated reduced locally closed substack $\calC_x \subseteq (\cathpolchdisps^{(m, 1 \blank \rdt)}_{g, J})_{\Fpbar}$ that satisfies $\abs{\calC_x} = \curlybr{x}$.
    The locally closed subscheme $\EKOR_x \subseteq (\calA_{g, J, N})_{\Fpbar}$ defined by the pullback square
    \[
    \begin{tikzcd}
        \EKOR_x \ar[r] \ar[d]
        & \roundbr*{\calA_{g, J, N}}_{\Fpbar} \ar[d]
        \\
        \calC_x \ar[r]
        & \roundbr*{\cathpolchdisps_{g, J}^{(m, 1 \blank \rdt)}}_{\Fpbar}
    \end{tikzcd}
    \]
    then is precisely the EKOR stratum corresponding to $x$, up to possibly carrying some non-reduced structure.

    Below in \Cref{thm:main-res} we will prove that the morphism $\upsilon^{(m, 1 \blank \rdt)}$ is smooth, and by \Cref{lem:restricting-pol-hpol} this implies the smoothness of $\calA_{g, J} \to \cathpolchdisps^{(m, n)}_{g, J}$.
    Observing that $\calC_x \to \Spec(\Fpbar)$ is a gerbe, hence smooth, we can then deduce the smoothness of $\EKOR_x$ over $\Fpbar$.
\end{remark}

\subsection{Smoothness of the morphism $\boldsymbol{\upsilon^{(m, n)}}$}

\begin{definition}
    Let $B_g$ be the finite partially ordered set of symmetric Newton polygons starting in $(0, 0)$, ending in $(2g, g)$ and with slopes in $[0, 1]$, where we declare $\nu \leq \nu'$ if $\nu$ lies above $\nu'$.
    This is precisely the set $B(\GSp(V), \mu_g)$, see \cite[Section 2]{he-rapoport}, and it classifies isocrystals of height $2g$ and slopes contained in $[0, 1]$ that are furthermore equipped with a symplectic form valued in the standard simple isocrystal of slope~$1$; see also \cite[Remark 3.4(iii)]{rapoport-richartz}.

    We have a natural map of sets $\abs{\catpolchpdiv_{g, J}} \to B_g$ that sends $((X_i)_i, (\rho_{i, j})_{i, j}, (\lambda_i)_i) \in \catpolchpdiv_{g, J}(k)$ for some algebraically closed field $k$ of characteristic $p$ to the Newton polygon of any of the $X_i$.
    We denote the fibers of this map by $S^{\pdiv}_{\nu}$ and their preimages in $\abs{\calA_{g, J, N}^{\wedge}}$ by $S_{\nu}$ and call these subsets \emph{Newton strata}.
\end{definition}

\begin{proposition} \label{thm:newton-strata}
    We have the following properties: 
    \begin{itemize}
        \item
        The $S_{\nu} \subseteq \abs{\calA_{g, J, N}^{\wedge}}$ are locally closed subsets.

        \item
        Let $\nu, \nu' \in B_g$.
        Then the intersection $\overline{S_{\nu'}} \cap S_{\nu}$ is non-empty if and only if $\nu \leq \nu'$.

        \item
        $\abs{\catpolchpdiv_{g, J}^{\formal}}$ is the union of those $S^{\pdiv}_{\nu}$ such that $\nu$ does not contain segments of slope $0$ or $1$.
    \end{itemize}
\end{proposition}

\begin{proof}
    This is all contained in \cite{he-rapoport}.
    See in particular Axiom 3.5, Theorem 5.6 and Section 7.
\end{proof}

\begin{proposition} \label{lem:specializing}
    Let $\nu, \nu' \in B_g$ with $\nu \leq \nu'$, and let $x \in S_{\nu'}^{\pdiv}$.
    Then there exists a preimage $y \in S_{\nu'}$ of $x$ that specializes to a point in $S_{\nu}$.
\end{proposition}

\begin{proof}
    By \Cref{thm:newton-strata} there exists a point in $S_{\nu'}$ specializing to a point in $S_{\nu}$.
    This specialization can be realized by a point
    \[
        A' = \roundbr*{\left(A'_i\right)_i, \left(\rho'_{i, j}\right)_{i, j}, \left(\lambda'_i\right)_i, \eta'} \in \calA_{g, J, N}(R)
    \]
    for some rank $1$ valuation ring $R$ of characteristic $p$ with algebraically closed fraction field $K$.
    Write $X' \coloneqq A'[p^{\infty}] \in \catpolchpdiv_{g, J}(R)$ for the image of $A'$.
    After possibly enlarging $R$ we also find an object
    \[
        X = \roundbr*{(X_i)_i, \left(\rho_{i, j}\right)_{i, j}, (\lambda_i)_i} \in \catpolchpdiv_{g, J}(K)
    \]
    representing $x$.

    Now $X'_K$ and $X$ both lie in $S^{\pdiv}_{\nu'}$; hence there exists an isomorphism between their associated isocrystals that is compatible with the polarizations.
    After multiplying such an isomorphism with $p^M$ for a suitable big enough integer $M \geq 0$, we obtain an isogeny
    \[
        f = (f_i)_i \colon X'_K \lra X
    \]
    in $\catchpdiv_{g, J}(K)$ that satisfies $f^* \lambda = p^{2M} \cdot \lambda'$.

    For $i \in J$ let $H_i \subseteq X'_i$ be the flat closure of $\ker(f_i) \subseteq X'_{i, K}$; this is a finite free $R$-subgroup scheme of $X'_i$ of order $p^{2gM}$.
    Define $A''_i \coloneqq A'_i/H_i$ and let $\rho''_{i, j} \colon A''_j \to A''_i$ be the isogeny induced by $\rho'_{i, j}$.
    Then there exists a unique isomorphism $\lambda''_i \colon A''_i \to (A''_{-i})^{\vee}$ that makes the diagram
    \[
    \begin{tikzcd}[column sep = large]
        A'_i \ar[r, "p^{2M} \lambda'_i"] \ar[d]
        & \roundbr*{A'_{-i}}^{\vee}
        \\
        A''_i \ar[r, "\lambda''_i"]
        & \roundbr*{A''_{-i}}^{\vee} \ar[u]
    \end{tikzcd}
    \]
    commutative.
    In this way we obtain a point
    \[
        A'' = \roundbr*{\left(A''_i\right)_i, \left(\rho''_{i, j}\right)_{i, j}, \left(\lambda''_i\right)_i, \eta'} \in \calA_{g, J, N}(R)
    \]
    (note that $A''_i[N] = A'_i[N]$ because $H_i$ is $p$-primary).

    By construction $A''$ still realizes a specialization from $S_{\nu'}$ to $S_{\nu}$, and moreover $f$ gives rise to an isomorphism $A''[p^{\infty}]_K \cong X$ in $\catpolchpdiv_{g, J}(K)$.
    Thus setting $y \in S_{\nu'} \subseteq \abs{\calA_{g, J, N}^{\wedge}}$ to be the point corresponding to $A''_K$ finishes the proof.
\end{proof}

\begin{corollary} \label{cor:specializing}
    Every point $x \in \abs{\catpolchpdiv_{g, J}}$ has a preimage $y \in \abs{\calA_{g, J, N}^{\wedge}}$ that specializes to a point mapping into $\abs{\catpolchpdiv_{g, J}^{\formal}}$.
\end{corollary}

\begin{proof}
    This follows from combining \Cref{lem:specializing} with the last part of \Cref{thm:newton-strata}.
\end{proof}

\begin{lemma} \label{lem:infinitesimal-lifting}
    Let $\calX$, $\calY$ be locally Noetherian formal algebraic stacks, let $f \colon \calX \to \calY$ be a morphism representable by algebraic stacks that is locally of finite type, and let $x \in \abs{\calX}$.

    Then $f$ is smooth at $x$ if and only if every lifting problem
    \[
    \begin{tikzcd}
        \Spec(B) \ar[r] \ar[d]
        & \calX \ar[d, "f"]
        \\
        \Spec(B') \ar[r] \ar[ru, dashed]
        & \calY\rlap{,}
    \end{tikzcd}
    \]
    where $B' \to B$ is a surjective homomorphism between Artinian local rings and the map $\abs{\Spec(B)} \to \abs{\calX}$ has image $\curlybr{x}$, admits a solution.
\end{lemma}

\begin{proof}
    When $\calX$ and $\calY$ are schemes, this is precisely \cite[Tag 02HX]{stacks-project}.
    The general case can be reduced to this; see also \cite[Tag 0DNV]{stacks-project}.
\end{proof}

\begin{theorem} \label{thm:main-res}
    For $n \neq 1 \blank \rdt$ the morphism $\upsilon^{(m, n)} \colon \calA_{g, J, N}^{\wedge} \to \catpolchdisps_{g, J}^{(m, n)}$ is smooth.
    Similarly, the morphism $\upsilon^{(m, 1 \blank \rdt)} \colon (\calA_{g, J, N})_{\Fp} \to \catpolchdisps_{g, J}^{(m, 1 \blank \rdt)}$ is also smooth.
\end{theorem}

\begin{proof}
    Assume that $n \neq 1 \blank \rdt$.
    The morphism $\upsilon^{(m, n)}$ factors as the composition
    \[
        \calA_{g, J, N}^{\wedge} \stackrel{\pi}\lra \catpolchpdiv_{g, J} \lra \catpolchdisps_{g, J} \lra \catpolchdisps_{g, J}^{(m, n)}.
    \]
    We have the following information: 
    \begin{itemize}
        \item
        $\pi \colon \calA_{g, J, N}^{\wedge} \to \catpolchpdiv_{g, J}$ is formally \'etale by the Serre--Tate theorem;  see \cite[Appendix]{drinfeld}.

        \item
        $\catpolchpdiv_{g, J} \to \catpolchdisps_{g, J}$ restricts to an isomorphism $\catpolchpdiv_{g, J}^{\formal} \to \catpolchdisps_{g, J}^{F \blank \nilp}$;  see \Cref{prop:disp-p-div-chains}.
        As $\catpolchdisps_{g, J}^{F \blank \nilp} \subseteq \catpolchdisps_{g, J}$ is stable under nilpotent thickenings, this implies that every lifting problem
        \[
        \begin{tikzcd}
            \Spec(B) \ar[r] \ar[d]
            & \catpolchpdiv_{g, J} \ar[d]
            \\
            \Spec(B') \ar[r] \ar[ru, dashed]
            & \catpolchdisps_{g, J}\rlap{,}
        \end{tikzcd}
        \]
        where $B' \to B$ is a surjective homomorphism of $p$-nilpotent rings with nilpotent kernel and the map $\Spec(B) \to \catpolchpdiv_{g, J}$ has image inside $\catpolchpdiv_{g, J}^{\formal}$, admits a (unique) solution.

        \item
        $\catpolchdisps_{g, J} \to \catpolchdisps_{g, J}^{(m, n)}$ is formally smooth; see \Cref{lem:polarized-disps-quotient-stack}.
    \end{itemize}
    Let $U \subseteq \abs{\calA_{g, J, N}^{\wedge}}$ be the smooth locus of $\upsilon^{(m, n)}$.
    Using \Cref{lem:infinitesimal-lifting} we can now deduce the following: 
    \begin{itemize}
        \item
        Every point $x \in \abs{\calA_{g, J, N}^{\wedge}}$ that maps into $\abs{\catpolchpdiv_{g, J}^{\formal}}$ is contained in $U$.

        \item
        $U$ is of the form $U = \pi^{-1}(V)$ for some $V \subseteq{\catpolchpdiv_{g, J}}$.
    \end{itemize}
    Now \Cref{cor:specializing}, together with the openness of $U \subseteq \abs{\calA_{g, J, N}^{\wedge}}$, implies that $V = \abs{\catpolchpdiv_{g, J}}$ and consequently that $U = \abs{\calA_{g, J, N}^{\wedge}}$ as desired.

    The case $n = 1 \blank \rdt$ now follows from \Cref{lem:disps-quotient-stack}.
\end{proof} 
    \appendix
\setcounter{section}{1}

\section*{Appendix. Categories with actions and dualities}
\addcontentsline{toc}{section}{Appendix. Categories with actions and dualities}

\setcounter{theorem}{0}

\subsection{Commutative monoids and actions}

Fix an ambient $(2, 1)$-category $\bsfC$ with finite $2$-products.
We write $\curlybr{\ast} \in \bsfC$ for an empty $2$-product.

\begin{definition}
    A \emph{commutative monoid in $\bsfC$} is an object $\calM \in \bsfC$ that is equipped with $1$-morphisms
    \[
        \otimes \colon \calM \times \calM \lra \calM, \quad 1 \colon \curlybr{\ast} \lra \calM
    \]
    and $2$-isomorphisms
    \[
         (a \otimes b) \otimes c \lra a \otimes (b \otimes c), \quad a \otimes b \lra b \otimes a, \quad 1 \otimes a \lra a
    \]
    between $1$-morphisms $\calM^i \to \calM$ for $i = 3, 2, 1$ such that the diagrams
    \[
    \begin{tikzcd}[column sep = tiny]
        ((a \otimes b) \otimes c) \otimes d \ar[d, equal] \ar[rr]
        & & (a \otimes (b \otimes c)) \otimes d \ar[rr]
        & & a \otimes ((b \otimes c) \otimes d) \ar[rr]
        & & a \otimes (b \otimes (c \otimes d)) \ar[d, equal]
        \\
        ((a \otimes b) \otimes c) \otimes d \ar[rrr]
        & & & (a \otimes b) \otimes (c \otimes d) \ar[rrr]
        & & & a \otimes (b \otimes (c \otimes d)) \rlap{,}
    \end{tikzcd}
    \]
    \[
    \begin{tikzcd}[column sep = small]
        a \otimes b \ar[rr, equal] \ar[rd]
        & & a \otimes b
        \\
        & b \otimes a\rlap{,} \ar[ru] &
    \end{tikzcd}
    \]
    \[
    \begin{tikzcd}
        (a \otimes b) \otimes c \ar[d, equal] \ar[r]
        & (b \otimes a) \otimes c \ar[r]
        & b \otimes (a \otimes c) \ar[r]
        & b \otimes (c \otimes a) \ar[d, equal]
        \\
        (a \otimes b) \otimes c \ar[r]
        & a \otimes (b \otimes c) \ar[r]
        & (b \otimes c) \otimes a \ar[r]
        & b \otimes (c \otimes a)\rlap{,}
    \end{tikzcd}
    \]
    \[
    \begin{tikzcd}[column sep = small]
        (1 \otimes a) \otimes b \ar[rr] \ar[rd]
        & & a \otimes b
        \\
        & 1 \otimes (a \otimes b) \ar[ru] &
    \end{tikzcd}
    \]
    are commutative.

    The commutative monoids in $\bsfC$ naturally form the class of objects of a $(2, 1)$-category that is given as follows.

    A $1$-morphism $F \colon \calM \to \calN$ between two commutative monoids in $\bsfC$ is a $1$-morphism $F \colon \calM \to \calN$ in $\bsfC$ that is equipped with $2$-isomorphisms $F(a \otimes b) \to F(a) \otimes F(b)$, $F(1) \to 1$ such that the diagrams
    \[
    \begin{tikzcd}
        F((a \otimes b) \otimes c) \ar[r] \ar[d]
        & F(a \otimes b) \otimes F(c) \ar[r]
        & (F(a) \otimes F(b)) \otimes F(c) \ar[d]
        \\
        F(a \otimes (b \otimes c)) \ar[r]
        & F(a) \otimes F(b \otimes c) \ar[r]
        & F(a) \otimes (F(b) \otimes F(c))\rlap{,}
    \end{tikzcd}
    \]
    \[
    \begin{tikzcd}
        F(a \otimes b) \ar[r] \ar[d]
        & F(a) \otimes F(b) \ar[d]
        \\
        F(b \otimes a) \ar[r]
        & F(b) \otimes F(a)\rlap{,}
    \end{tikzcd}
    \]
    \[
    \begin{tikzcd}
        F(1 \otimes a) \ar[r] \ar[rd]
        & F(1) \otimes F(a) \ar[r]
        & 1 \otimes F(a) \ar[ld]
        \\
        & F(a) &
    \end{tikzcd}
    \]
    are commutative.

    A $2$-isomorphism $\alpha \colon F \to G$ between two $1$-morphisms $F, G \colon \calM \to \calN$ as above is a $2$-isomorphism $\alpha \colon F \to G$ in $\bsfC$ such that the diagrams
    \[
    \begin{tikzcd}
        F(a \otimes b) \ar[r, "\alpha"] \ar[d]
        & G(a \otimes b) \ar[d]
        \\
        F(a) \otimes F(b) \ar[r, "\alpha \otimes \alpha"]
        & G(a) \otimes G(b)\rlap{,}
    \end{tikzcd}
    \]
    \[
    \begin{tikzcd}[column sep = small]
        F(1) \ar[rr, "\alpha"] \ar[rd]
        & & G(1) \ar[ld]
        \\
        & 1 &
    \end{tikzcd}
    \]
    are commutative.
\end{definition}

\begin{definition} \label{def:cat-action}
    Let $\calM$ be a commutative monoid in $\bsfC$.
    An \emph{$\calM$-object in $\bsfC$} is an object $\calC \in \bsfC$ that is equipped with a $1$-morphism
    \[
        \otimes \colon \calM \times \calC \lra \calC
    \]
    and $2$-isomorphisms
    \[
        (a \otimes b) \otimes x \lra a \otimes (b \otimes x), \quad 1 \otimes x \lra x
    \]
    such that the diagrams
    \[
    \begin{tikzcd}[column sep = tiny]
        ((a \otimes b) \otimes c) \otimes x \ar[d, equal] \ar[rr]
        & & (a \otimes (b \otimes c)) \otimes x \ar[rr]
        & & a \otimes ((b \otimes c) \otimes x) \ar[rr]
        & & a \otimes (b \otimes (c \otimes x)) \ar[d, equal]
        \\
        ((a \otimes b) \otimes c) \otimes x \ar[rrr]
        & & & (a \otimes b) \otimes (c \otimes x) \ar[rrr]
        & & & a \otimes (b \otimes (c \otimes x))\rlap{,}
    \end{tikzcd}
    \]
    \[
    \begin{tikzcd}[column sep = small]
        (1 \otimes a) \otimes x \ar[rr] \ar[rd]
        & & a \otimes x
        \\
        & 1 \otimes (a \otimes x) \ar[ru] &
    \end{tikzcd}
    \]
    are commutative.

    The tuples $(\calM, \calC)$ consisting of a commutative monoid $\calM$ in $\bsfC$ and an $\calM$-object $\calC$ in $\bsfC$ naturally form the class of objects of a $(2, 1)$-category that is given as follows.

    A $1$-morphism $(F, F') \colon (\calM, \calC) \to (\calN, \calD)$ consists of a $1$-morphism $F \colon \calM \to \calN$ of commutative monoids in $\bsfC$ and a $1$-morphism $F' \colon \calC \to \calD$ in $\bsfC$ that is equipped with a $2$-isomorphism $F'(a \otimes x) \to F(a) \otimes F'(x)$ such that the diagrams
    \[
    \begin{tikzcd}
        F'((a \otimes b) \otimes x) \ar[r] \ar[d]
        & F(a \otimes b) \otimes F'(x) \ar[r]
        & (F(a) \otimes F(b)) \otimes F'(x) \ar[d]
        \\
        F'(a \otimes (b \otimes x)) \ar[r]
        & F(a) \otimes F'(b \otimes x) \ar[r]
        & F(a) \otimes (F(b) \otimes F'(x))\rlap{,}
    \end{tikzcd}
    \]
    \[
    \begin{tikzcd}
        F'(1 \otimes x) \ar[r] \ar[rd]
        & F(1) \otimes F'(x) \ar[r]
        & 1 \otimes F'(x) \ar[ld]
        \\
        & F'(x) &
    \end{tikzcd}
    \]
    are commutative.

    A $2$-isomorphism $(\alpha, \alpha') \colon (F, F') \to (G, G')$ between two $1$-morphisms $(F, F'), (G, G') \colon (\calM, \calC) \to (\calN, \calD)$ as above consists of a $2$-isomorphism $\alpha \colon F \to G$ between $1$-morphisms of commutative monoids in $\bsfC$ and a $2$-isomorphism $\alpha' \colon F' \to G'$ in $\bsfC$ such that the diagram
    \[
    \begin{tikzcd}
        F'(a \otimes x) \ar[r, "\alpha'"] \ar[d]
        & G'(a \otimes x) \ar[d]
        \\
        F(a) \otimes F'(x) \ar[r, "\alpha \otimes \alpha'"]
        & G(a) \otimes G'(x)
    \end{tikzcd}
    \]
    is commutative.
\end{definition}

\subsection{Categories with duality}

\begin{definition} \label{def:cat-duality}
    A \emph{category with duality} is a category $\calC$ that is equipped with an equivalence of categories
    \[
        \calC^{\op} \lra \calC, \quad x \longmapsto x^{\vee}
    \]
    and a natural isomorphism
    \[
        x \lra x^{\vee \vee}
    \]
    between functors $\calC \to \calC$ such that the diagram
    \[
    \begin{tikzcd}[column sep = small]
        x^{\vee} \ar[rr, equal] \ar[rd]
        & & x^{\vee}
        \\
        & x^{\vee \vee \vee} \ar[ru] &
    \end{tikzcd}
    \]
    is commutative (where the sloped arrow on the right is the image of the isomorphism $x \to x^{\vee \vee}$ under the duality equivalence above).

    A functor $F \colon \calC \to \calD$ between categories with duality is a functor $F \colon \calC \to \calD$ that is equipped with a natural isomorphism $F(x^{\vee}) \to F(x)^{\vee}$ such that the diagram
    \[
    \begin{tikzcd}
        & F(x) \ar[dl] \ar[dr] &
        \\
        F(x^{\vee \vee}) \ar[r]
        & F(x^{\vee})^{\vee} \ar[r]
        & F(x)^{\vee \vee}
    \end{tikzcd}
    \]
    is commutative (where the second horizontal arrow denotes the inverse of the dual of the isomorphism $F(x^{\vee}) \to F(x)^{\vee}$).

    A natural isomorphism $\alpha \colon F \to G$ between two functors $F, G \colon \calC \to \calD$ as above is a natural isomorphism $\alpha \colon F \to G$ such that the diagram
    \[
    \begin{tikzcd}
        F(x^{\vee}) \ar[r, "\alpha"] \ar[d]
        & G(x^{\vee}) \ar[d]
        \\
        F(x)^{\vee} \ar[r, "\alpha^{\vee, -1}"]
        & G(x)^{\vee}
    \end{tikzcd}
    \]
    is commutative.

    In this way the categories with duality naturally form the class of objects of a $(2, 1)$-category.
\end{definition}

\begin{remark} \label{rmk:cat-action-duality}
    The $(2, 1)$-category of (essentially small) categories with duality has finite $2$-products, and the natural forgetful $2$-functor
    \[
        \curlybr*{\text{categories with duality}} \lra \curlybr*{\text{categories}}
    \]
    preserves finite $2$-products.

    Thus we obtain the notions of a \emph{symmetric monoidal category with duality} and an \emph{action of such a symmetric monoidal category with duality on a category with duality}.
\end{remark}

\begin{remark} \label{rmk:cat-rigid-duality}
    Every rigid symmetric monoidal category naturally carries a duality.
    Thus we obtain a natural $(2, 1)$-functor
    \[
        \curlybr*{\text{rigid symmetric monoidal categories}} \lra \curlybr*{\text{symmetric monoidal categories with duality}}.
    \]
\end{remark}

\begin{remark} \label{rmk:cat-lambda}
    Let $\Lambda$ be a commutative ring.
    Then the notions of a category with duality, a symmetric monoidal category (with duality) and an action of a symmetric monoidal category (with duality) on a category (with duality) have obvious $\Lambda$-linear variants that are obtained by requiring all appearing functors to be $\Lambda$-(multi-)linear.
\end{remark}

\begin{definition}
    Let $\calC$ be a $\ZZ$-linear category with duality.
    Then a morphism $f \colon x \to x^{\vee}$ in $\calC$ is called \emph{antisymmetric} if the diagram
    \[
    \begin{tikzcd}[column sep = small]
        x \ar[rr] \ar[rd, swap, "-f"]
        & & x^{\vee \vee} \ar[ld, "f^{\vee}"]
        \\
        & x^{\vee} &
    \end{tikzcd}
    \]
    is commutative.
    We define the groupoid $\catpol(\calC)$ of \emph{polarized objects in $\calC$} as the groupoid of tuples $(x, \lambda)$ with $x \in \calC$ and $\lambda \colon x \to x^{\vee}$ an antisymmetric isomorphism, where isomorphisms $(x, \lambda) \to (y, \zeta)$ are given by isomorphisms $x \to y$ in $\calC$ such that the diagram
    \[
    \begin{tikzcd}
        x \ar[r] \ar[d, "\lambda"]
        & y \ar[d, "\zeta"]
        \\
        x^{\vee} \ar[r]
        & y^{\vee}
    \end{tikzcd}
    \]
    is commutative (where the lower horizontal arrow is the inverse of the dual of $x \to y$).

    Now suppose that $\calC$ is equipped with an action of a $\ZZ$-linear rigid symmetric monoidal category $\calM$ (see \Cref{rmk:cat-rigid-duality}).
    Then a morphism $f \colon x \to a \otimes x^{\vee}$ in $\calC$ with $a \in \calM$ invertible is called \emph{antisymmetric} if the diagram
    \[
    \begin{tikzcd}
        x \ar[r] \ar[rd, swap, "-f"]
        & x^{\vee \vee} \ar[r]
        & a \otimes (a \otimes x^{\vee})^{\vee} \ar[ld, "\id_a \otimes f^{\vee}"]
        \\
        & a \otimes x^{\vee} &
    \end{tikzcd}
    \]
    is commutative.
    In the case $a = 1$ this recovers the first definition above.
    We define the groupoid $\cathpol(\calC) = \cathpol^{\calM}(\calC)$ of \emph{$(\calM$-$)$homogeneously polarized objects in $\calC$} as the groupoid of tuples $(x, a, \lambda)$ with $x \in \calC$, $a \in \calM$ invertible and $\lambda \colon x \to a \otimes x^{\vee}$ an antisymmetric isomorphism, where isomorphisms $(x, a, \lambda) \to (y, b, \zeta)$ are given by tuples of isomorphisms $x \to y$ in $\calC$ and $a \to b$ in $\calM$ such that the diagram
    \[
    \begin{tikzcd}
        x \ar[r] \ar[d, "\lambda"]
        & y \ar[d, "\zeta"]
        \\
        a \otimes x^{\vee} \ar[r]
        & b \otimes y^{\vee}
    \end{tikzcd}
    \]
    is commutative.
\end{definition}

\begin{lemma} \label{lem:pol-hpol-pullback}
    Let $\calC$ be a $\ZZ$-linear category with duality that is equipped with an action of a $\ZZ$-linear rigid symmetric monoidal category $\calM$.
    Then there is a natural $2$-Cartesian diagram of groupoids
    \[
    \begin{tikzcd}
        \catpol(\calC) \ar[r] \ar[d]
        & \cathpol(\calC) \ar[d]
        \\
        \curlybr{\ast} \ar[r]
        & \calM^{\simeq}\rlap{,}
    \end{tikzcd}
    \]
    where $\calM^{\simeq}$ denotes the groupoid core of $\calM$.
\end{lemma}

\begin{proof}
    The upper horizontal arrow is given by $(x, \lambda) \mapsto (x, 1, \lambda)$, the right vertical arrow by $(x, a, \lambda) \mapsto a$ and the lower horizontal arrow by $\ast \mapsto 1$.
    The diagram is in fact strictly commutative, so  we can use the identity as a commutativity constraint.

    Now the diagram is strictly Cartesian, and the right vertical arrow is a fibration (both claims can be checked by a direct computation), so  the diagram is also $2$-Cartesian.
\end{proof}

%%%%%%%%%%%%%%%%%%%%%
% References
%%%%%%%%%%%%%%%%%%%%%

\providecommand{\bysame}{\leavevmode\hbox to3em{\hrulefill}\thinspace}


\begin{thebibliography}{AGLR22+++}

  
\bibitem[AGLR22]{p-adic-local-models}
J.~Ansch\"utz, I.~Gleason, J.~Louren\c{c}o and T.~Richarz, \emph{On the $p$-adic
  theory of local models},  preprint \arXiv{2201.01234} (2022).

\bibitem[BS17]{bhatt-scholze}
B.~Bhatt and P.~Scholze, \emph{Projectivity of the {W}itt vector affine
  {G}rassmannian}, Invent.\ Math.~\textbf{209} (2017), no.~2,
329--423, 
\doi{10.1007/s00222-016-0710-4}.

\bibitem[BH20]{bueltel-hedayatzadeh}
O.~B\"{u}ltel and M.\,H.~Hedayatzadeh, \emph{{$(G,\mu)$}-Windows and Deformations of {$(G,\mu)$}-Displays},  preprint \arXiv{2011.09163} (2020).

\bibitem[BP18]{bueltel-pappas}
O.~B\"ultel and G.~Pappas, \emph{$(G, \mu)$-displays and Rapoport--Zink spaces},
  J.~Inst.\ Math.\ Jussieu \textbf{19} (2018), no.~4, 1211--1257,
  \doi{10.1017/S1474748018000373}.

\bibitem[Dan21]{daniels-tannakian-displays}
P.~Daniels, \emph{A Tannakian Framework for {$G$}-displays and Rapoport–Zink
Spaces}, Int.\ Math.\ Res.\ Not.\ IMRN \textbf{2021}, no.~22, 16963--17024, \doi{10.1093/imrn/rnz311}.

\bibitem[Dri76]{drinfeld}
V.\,G.~Drinfel'd, \emph{Coverings of $p$-adic symmetric regions}, Funct.\ Anal.\ Its Appl.~\textbf{10} (1976), no.~2, 107--115, 
\doi{10.1007/BF01077936}.

\bibitem[G\"or01]{goertz-flatness-local-models-linear}
U.~G\"ortz, \emph{On the flatness of models of certain Shimura varieties of
PEL-type}, Math.\ Ann.~\textbf{321} (2001), no.~3, 689--727,
\doi{10.1007/s002080100250}.

\bibitem[G\"or03]{goertz-flatness-local-models-symplectic}
\bysame, \emph{On the flatness of local models for the symplectic group},
Adv.\ Math.~\textbf{176} (2003), no.~1, 89--115,
\doi{10.1016/S0001-8708(02)00062-2}.

\bibitem[HR16]{he-rapoport}
X.~He and M.~Rapoport, \emph{Stratifications in the reduction of Shimura
varieties}, Manuscripta Math.~\textbf{152} (2016), no.~3-4, 317--343,
\doi{10.1007/s00229-016-0863-x}.
  
\bibitem[Hes20]{hesse}
J.~Hesse, \emph{EKOR strata on Shimura varieties with parahoric reduction},
   preprint \arXiv{2003.04738} (2020).

\bibitem[dJo93]{de-jong}
A.\,J.~de~Jong, \emph{The moduli spaces of principally polarized abelian
  varieties with $\Gamma_0(p)$-level structure}, J.~Algebraic Geom.\
  \textbf{2} (1993), no.~4, 667--688.

\bibitem[Kis10]{kisin-integral-models}
M.~Kisin, \emph{Integral models for Shimura varieties of abelian type}, J.~Amer.\ Math.\ Soc.~\textbf{23} (2010), no.~4, 967--1012,
\doi{10.1090/S0894-0347-10-00667-3}.

\bibitem[KP18]{kisin-pappas-integral-models}
M.~Kisin and G.~Pappas, \emph{Integral models of Shimura varieties with
  parahoric level structure}, Publ.\ Math.\ Inst.\ Hautes \'Etudes Sci.~\textbf{128} (2018), 121--218,
\doi{10.1007/s10240-018-0100-0}. 

\bibitem[KR00]{kottwitz-rapoport}
R.~Kottwitz and M.~Rapoport, \emph{Minuscule alcoves for ${\rm GL}_n$ and $G{\rm Sp}_{2n}$}, Manuscripta Math.~\textbf{102} (2000), no.~4, 403--428,
\doi{10.1007/s002290070034}.

\bibitem[Lau13]{lau-10}
E.~Lau, \emph{Smoothness of the truncated display functor}, J.~Amer.\ Math.\ Soc.~\textbf{26} (2013), no.~1, 129--165,
\doi{10.1090/S0894-0347-2012-00744-9}.

\bibitem[Lau21]{lau-higher-frames}
\bysame, \emph{Higher frames and $G$-displays}, Algebra Number Theory
\textbf{15} (2021), no.~9, 2315--2355,
\doi{10.2140/ant.2021.15.2315}. 

\bibitem[LZ18]{lau-zink-18}
E.~Lau and T.~Zink, \emph{Truncated Barsotti--Tate groups and displays}, J.~Inst.\ Math.\ Jussieu \textbf{17} (2018), no.~3, 541--581,
\doi{10.1017/S1474748016000116}.

\bibitem[MW04]{moonen-wedhorn}
B.~Moonen and T.~Wedhorn, \emph{Discrete invariants of varieties in positive
  characteristic}, Int.\ Math.\ Res.\ Not.~\textbf{2004}, 
no.~72, 3855--3903,
\doi{10.1155/S1073792804141263}. 

\bibitem[MFK94]{mumford-git}
D.~Mumford, J.~Fogarty and F.~Kirwan, \emph{Geometric invariant theory}, 3rd
  ed., Ergeb.\ Math.\ Grenzgeb.~(2), vol.~34,
  Springer-Verlag, Berlin, 1994,
  \doi{10.1007/978-3-642-57916-5}.

\bibitem[Oda69]{oda}
T.~Oda, \emph{The first de Rham cohomology group and Dieudonn\'e modules}, 
Ann.\ Sci.\ \'Ecole Norm.\ Sup.~(4) \textbf{2} (1969), no.~1,   63--135,
  \doi{10.24033/asens.1175}. 

\bibitem[Oor01]{oort-eo}
F.~Oort, \emph{A Stratification of a Moduli Space of Abelian Varieties}, In: \emph{Moduli of abelian varieties} (Texel Island, 1999; C.~Faber, G.~van~der Geer and F.~Oort, eds), pp.~345--416, Progr.\ Math., vol.~195, Birkh\"auser Verlag, Basel, 2001,
  \doi{10.1007/978-3-0348-8303-0\_13}.

\bibitem[Oor04]{oort-foliations}
\bysame, \emph{Foliations in moduli spaces of abelian varieties}, J.~Amer.\ Math.\ Soc.~\textbf{17} (2004), no.~2, 267--296,
\doi{10.1090/S0894-0347-04-00449-7}.

\bibitem[Pap23]{pappas-parahoric-disps}
G.~Pappas, \emph{On integral models of Shimura varieties}, Math.\ Ann.\textbf{385} (2023), no.~3-4, 2037-–2097, 
\doi{10.1007/s00208-022-02387-8}.

\bibitem[PWZ15]{pink-wedhorn-ziegler}
R.~Pink, T.~Wedhorn and P.~Ziegler, \emph{$F$-zips with additional structure},
Pacific J.~Math.~\textbf{274} (2015), no.~1, 183--236,
\doi{10.2140/pjm.2015.274.183}.

\bibitem[RR96]{rapoport-richartz}
M.~Rapoport and M.~Richartz, \emph{On the classification and specialization of
  $F$-isocrystals with additional structure}, Compos.\ Math.\
  \textbf{103} (1996), no.~2, 153--181. 

\bibitem[RZ96]{rapoport-zink-96}
M.~Rapoport and T.~Zink, \emph{Period Spaces for $p$-divisible Groups}, Ann.\ of
Math.\ Stud., vol.~141, Princeton Univ.\ Press, Princeton, NJ, 1996,
\doi{10.1515/9781400882601}.

\bibitem[SYZ21]{shen-yu-zhang}
X.~Shen, C.\,F.~Yu and C.~Zhang, \emph{EKOR strata for Shimura varieties with
  parahoric level structure}, Duke Math.~J.~\textbf{170} (2021), no.~14, 3111--3236,
\doi{10.1215/00127094-2021-0047}.

\bibitem[SYZ24]{shen-yu-zhang-errata}
\bysame, \emph{Errata for ``EKOR strata for Shimura varieties with parahoric
  level structure''}, preprint (2024), available at  
  \url{http://www.mcm.ac.cn/people/members/202012/t20201207_599977.html}.

\bibitem[SZ22]{shen-zhang-eo}
X.~Shen and C.~Zhang, \emph{Stratifications in good reductions of Shimura
varieties of abelian type}, Asian J.~Math.~\textbf{26} (2022), no.~2, 167--226, \doi{10.4310/ajm.2022.v26.n2.a2}.

\bibitem[{Sta}18]{stacks-project}
The Stacks Project Authors, \emph{Stacks project},
\url{https://stacks.math.columbia.edu}, 2018.


\bibitem[Vas08]{vasiu-2}
A.~Vasiu, \emph{Level {$m$} stratifications of versal deformations of
  {$p$}-divisible groups}, J.~Algebraic Geom.~\textbf{17} (2008),
no.~4, 599--641,
\doi{10.1090/S1056-3911-08-00495-5}.

\bibitem[VW13]{viehmann-wedhorn}
E.~Viehmann and T.~Wedhorn, \emph{Ekedahl--Oort and Newton strata for Shimura
  varieties of PEL type}, Math.\ Ann.~\textbf{356} (2013), no.~4, 
1493--1550,
\doi{10.1007/s00208-012-0892-z}.

\bibitem[XZ17]{xiao-zhu}
L.~Xiao and X.~Zhu, \emph{Cycles on Shimura varieties via geometric Satake},
  preprint \arXiv{1707.05700} (2017).

\bibitem[Zha18]{zhang-eo}
C.~Zhang, \emph{Ekedahl--Oort Strata for Good Reductions of Shimura Varieties of
Hodge Type}, Canad.~J.\ Math.~\textbf{70} (2018), no.~2,  451--480, \doi{10.4153/CJM-2017-020-5}.


\bibitem[Zin02]{zink-02}
T.~Zink, \emph{The display of a formal $p$-divisible group}, In: \emph{Cohomologies $p$-adiques et applications arithm\'etiques, I} (P.~Berthelot, J.\,M. Fontaine,  L.~Illusie, K.~Kato and M.~Rapoport, eds), pp.~127--248, Ast\'erisque, vol.~278 (2002).

\end{thebibliography}
\end{document}